\documentclass[12pt]{article}
\usepackage[top=1in, bottom=1in, left=1in, right=1in]{geometry}
\usepackage{amsmath}
\usepackage{amsfonts}
\usepackage{amsthm}
\usepackage{amssymb, setspace}
\usepackage{color,graphicx, epsfig, geometry, hyperref,fancyhdr}
\newtheorem{theorem}{Theorem}[section]
\usepackage{framed}
\usepackage{float} 
\newtheorem{definition}{Definition}[section] 
\newtheorem{remark}{Remark}[section] 
 
\newtheorem{lemma}{Lemma}[section]

\newcommand{\R}{\mathbb{R}}
\newcommand{\C}{\mathbb{C}}

\newcommand{\grad}{\nabla}

\newcommand{\weakc}{\rightharpoonup}

\newcommand{\eps}{\varepsilon}

\graphicspath{{pics/}}

\begin{document}
\setlength{\parskip}{1mm}
\setlength{\oddsidemargin}{0.1in}
\setlength{\evensidemargin}{0.1in}
\lhead{}
\rhead{}

\begin{center}
{\bf HOMOGENIZATION OF THE TRANSMISSION EIGENVALUE PROBLEM FOR PERIODIC MEDIA AND APPLICATION TO THE INVERSE PROBLEM}
\end{center}
\vspace{0.05in}

\begin{center}
Fioralba Cakoni\\
Department of Mathematical Sciences\\
University of Delaware Newark\\
Delaware 19716-2553, USA \\ 
E-mail address: cakoni@math.udel.edu\\
\vspace{0.3in}

Houssem Haddar \\
INRIA Saclay Ile de France/CMAP Ecole Polytechnique \\
Route de Saclay, 91128 Palaiseau Cedex, France\\
E-mail address: haddar@cmap.polytechnique.fr\\
\vspace{0.3in}

Isaac Harris \\
Department of Mathematical Sciences\\
University of Delaware Newark\\
Delaware 19716-2553, USA \\
E-mail address: iharris@udel.edu
\end{center}

\begin{abstract} We consider the interior transmission  problem associated with the scattering by an inhomogeneous (possibly anisotropic) highly oscillating periodic media. We show that,  under appropriate assumptions,  the solution of the interior transmission problem converges to the solution of a homogenized problem as the period goes to zero. Furthermore,  we prove that the associated real transmission eigenvalues  converge to  transmission eigenvalues of the homogenized problem. Finally we show how to use the first transmission eigenvalue of the period media, which is measurable from  the scattering data, to obtain information about constant effective material properties of the periodic media. The convergence results presented here are not optimal. Such results with rate of convergence involve the analysis of the boundary correction and will be subject of a forthcoming paper.
 \end{abstract}

\section{Introduction}
We consider the  transmission eigenvalue problem associated with the scattering by inhomogeneuos (possibly anisotropic) highly oscillating periodic media in the frequency domain. The governing equations possess rapidly oscillating periodic coefficients  which typically model  the wave propagation through composite  materials with  fine microstructure. Such composite materials are at the foundation of many contemporary engineering designs and are used to produce materials with special properties by combining in a particular structure (usually in periodic patterns) different materials.  In practice, it is desirable to understand these special properties, in particular macrostructure behavior of the composite materials which mathematically is achievable by using homogenization approach \cite{allair}, \cite{BLP}.  Our concern here is with the study of the corresponding transmission eigenvalues, in particular their behavior as the period in the medium approaches zero. To this end, it is essential to prove strong $H^1(D)$-convergence  of the resolvent corresponding to the transmission eigenvalue problem, or as known as the  solution of the interior transmission problem.  Transmission eigenvalues associated with the scattering problem for an inhomogeneous media are closely related to the so-called non scattering  frequencies  \cite{corner}, \cite{CC-book}, \cite{CakHad2}.  Such eigenvalues can be determined from  scattering data \cite{blow}, \cite{kirsch} and provide information about material properties of the scattering media \cite{cakginhad}, and hence can be used to estimate the refractive index of the media. In particular, in the current work we use the first transmission eigenvalue to estimate the effective material properties of the periodic media.

More precisely, let $D \subset \R^d$ be a bounded simply connected open set with piecewise smooth  boundary $\partial D$ representing the support of the inhomogeneous periodic media. Let $\epsilon>0$ be the length of the period, which is assumed to be very small in comparison to the size of $D$ and let $Y=(0,\,1)^d$ be the rescaled unit periodic cell. We assume that the constitutive material properties in the media are given by  a  positive definite symmetric matrix valued function $A_\epsilon:=A(x/\epsilon)\in  L^{\infty} \left(D, \R^{d \times d} \right)$  and a positive function $n_\epsilon:=n(x/\epsilon) \in L^{\infty} \left(D \right)$. Furthermore, assume that both $A(y)$ and $n(y)$ are periodic in $y=x/\epsilon$ with period $Y$ (here $x\in D$ is refer to as the slow variable where $y=x/\epsilon\in {\mathbb R}^d$ is referred to as the fast variable). We remark that our  convergence analysis is also  valid in the absorbing case, i.e. for complex valued $A$ and $n$, but since the real eigenvalues (which are the measurable ones) exist only for real valued material properties,  we limit ourselves to this case. Let us introduce the following notations:
\begin{eqnarray}
\inf_{y \in Y} \inf_{|\xi|=1} \overline{\xi} \cdot A(y) \xi =A_{min}>0 \, \, &\textrm{ and }& \, \, \sup_{y \in Y} \sup_{|\xi|=1} \overline{\xi} \cdot A(y) \xi =A_{max}< \infty\label{c0} \\
\inf_{y \in Y} n(y)=n_{min}>0 \, \,&\textrm{ and }&  \, \, \sup_{y \in Y} n(y)=n_{max} < \infty. \label{c1}
\end{eqnarray}
The interior transmission eigenvalue problem for the anisotropic media ($d=2$ in electromagnetic scattering and $d=3$ in acoustic scattering) reads: find $(w_\eps,v_\eps)$ satisfying:
\begin{eqnarray}
\grad \cdot A_\epsilon\grad w_\eps +k_\eps^2 n_\epsilon w_\eps=0 \,  &\textrm{ in }& \,  D\label{def1} \\
 \Delta v_\eps + k_\eps^2 v_\eps=0  &\textrm{ in }& \,  D \\
 w_\eps=v_\eps  &\textrm{ on }& \partial D\\
 \frac{\partial w_\eps}{\partial \nu_{A_\eps}}=\frac{\partial v_\eps}{\partial \nu} &\textrm{ on }& \partial D \label{def2}
\end{eqnarray}
where $\frac{\partial w}{\partial \nu_{A}} =\nu \cdot A\nabla w$. Note that the spaces for the solution  $(w_\eps,v_\eps)$  will become precise later  since they depend on whether $A=I$ or $A\neq I$.
\begin{figure}[ht!]
\centering
\includegraphics[scale=0.44]{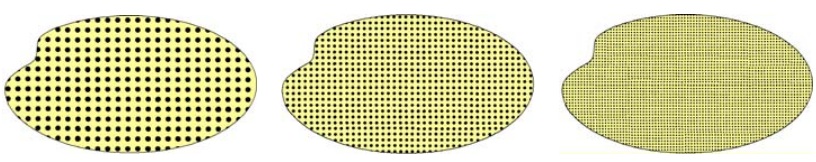}\\
\caption{A periodic domain for three different values of $\epsilon$.}
\label{areatev}
\end{figure}

\begin{definition}
The  values $k_\eps \in \C$ for which (\ref{def1})-(\ref{def2}) has a nontrivial solution are called transmission eigenvalues. The corresponding nonzero solutions $(w_\eps,v_\eps)$  are referred to eigenfunctions. 
\end{definition}
It is known that, provided that $A-I$ or/and $n-1$  do not change sign in $D$ and are bounded away from zero, the real transmission eigenvalues exist \cite{cakginhad}, \cite{cakonikirsch}, \cite{harris}. However the transmission eigenvalue problem is non-selfajoint and this causes complications in the analysis. In this study we are interested in the behavior of eigenvalues $k_\eps$ and eigenfunctions $(w_\eps,v_\eps)$ in limiting case as $\eps \rightarrow 0$. In particular we will be interested in the limit of the {\it real transmission eigenvalues} since they  have been proven to exists and can be determined from scattering data.  
\subsection{Formal Asymptotic Expansion}  We are interested in developing the asymptotic theory of (\ref{def1})-(\ref{def2}) as the period size $\epsilon\to 0$. To this end we need to define the space 
$$H^1_{\#}(Y):=\{ u \in H^1(Y) \, | \,  u(y) \textrm{ is } Y \textrm{-periodic} \}$$ 
and consider the subspace of $Y$-periodic $H^1$-functions of mean zero, i.e. 
$$\widehat{H}^1_{\#}(Y):=\left\{ u \in H^1_{\#}(Y) \, | \, \int_Y u(y) \, dy=0 \right\}.$$
One expects (as our convergence analysis will confirm) that the homogenized or limiting transmission eigenvalue problem will be
\begin{eqnarray}
\grad_x \cdot A_h \grad_x w_0 +k^2 n_h w_0=0 \,  &\textrm{ in }& \,  D\label{h1}\\
 \Delta_x v_0 + k^2 v_0=0  &\textrm{ in }& \,  D \\
w_0=v_0   &\textrm{ on }& \partial D\\
 \frac{\partial w_0}{\partial \nu_{A_h}}=\frac{\partial v_0}{\partial \nu}     &\textrm{ on }& \partial D \label{h2}.
\end{eqnarray}
where 
\begin{equation}\label{hr1}
A_h= \frac{1}{|Y|} \int_Y \left(A(y)-A(y) \grad_y \vec{\psi}(y) \right) \, dy\qquad \mbox{and}\qquad  n_h=\frac{1}{|Y|} \int_Y n(y) \, dy,
\end{equation}
The so-called cell function  $\psi_i(y) \in \widehat{H}^1_{\#}(Y)$ is the unique solution to
\begin{equation}\label{psi}
\grad_y \cdot A \grad_y \psi_i=\grad_y \cdot A e_i \, \textrm{ in } \,  Y,
\end{equation}
where $e_i$ is the $i$-th standard basis vector in $\R^d$. 
We recall that it is well known that the homogenized (otherwise known as effective) anisotropic constitutive parameter of the periodic medium $A_h$ satisfies the following estimates \cite{allair}
\begin{equation}\label{estA}
\hspace*{0.5cm} \left(\frac{1}{|Y|}\int\limits_YA^{-1}(y)dy\right)^{-1}\hspace*{-0.2cm}\xi\cdot \overline{\xi}\leq A_h\xi\cdot\overline{\xi}\leq \left(\frac{1}{|Y|}\int\limits_YA(y)dy\right)\xi\cdot \overline{\xi}\quad  \xi\in {\mathbb C}^d
\end{equation}
hence (\ref{c0}) and (\ref{c1}) are also satisfied for $A_h$ and $n_h$. 

The question now  is whether the eigenvalues $k_\epsilon$ and corresponding eigenfunctions $v_\epsilon, w_\epsilon$ of (\ref{def1})-(\ref{def2}) converge to eigenvalues and eigenfuctions of (\ref{h1})-(\ref{h2}).  For the Dirichlet and Neumann  eigenvalue problem for periodic structures the question of convergence is studied in  details. In particular for these problems,  the convergence is proven in  \cite{BLP}, \cite{kesavan1} and  \cite{kesavan2} and  the rate of convergence  with explicit first order correction involving the  boundary layer is studied in  \cite{kenig}, \cite{kenig2}, \cite{shari1}, \cite{shari2} and \cite{santosa}.
Given the peculiarities of the transmission eigenvalue problem such as non-selfadjointness and the lack of ellipticity, the above approaches cannot be applied in a straightforward manner. Furthermore the transmission eigenvalue problem exhibits different properties  in the case when $A\neq I$ or $A=I$ , hence each of these cases need to be studied separately \cite{CakHad2}. We remark that the existence of an infinite set of transmission eigenvalues in general settings  is proven in \cite{lak1}, \cite{lak2} and \cite{robbiano}, where the existence of an infinite set of real transmission eigenvalues along with monotonicity properties  are proven in \cite{cakginhad} and \cite{cakonikirsch}.  In the next section we justify the formal asymptotic  for the resolvent corresponding to the transmission eigenvalue problem using the two scale convergent approach developed in \cite{allaire1}. This is followed by the proof of convergence results for a subset of real transmission eigenvalues in Section 3. The last section is dedicated to some preliminary numerical examples where we investigate convergence properties of the first transmission eigenvalue and demonstrate the  feasibility of using the first real transmission eigenvalue to determine the effective material properties  $A_h$ and $n_h$.  

\section{Convergence Analysis}
We start with studying the convergence of the resolvent of the transmission eigenvalue problem, i.e. of the solution to the interior transmission problem with source terms. The approach to study the interior transmission problem depends on the fact whether 
$A(y)\neq I$  for all $y\in Y$ or $A(y) \equiv I$. 
\subsection{The case of $A_\eps \neq I$} \label{sec3.1} We assume that  $A_{min}>1$ or $A_{max}<1$  in addition to (\ref{c0}) and (\ref{c1}). For $f_\epsilon$ and $g_\epsilon$ in $L^2(D)$ strongly convergent to $f$ and $g$, respectively,  as $\epsilon \to 0$ we consider the interior transmission problem  of finding $(w_\eps , v_\eps )\in H^1(D) \times H^1(D)$ such that 
\begin{eqnarray}
\grad \cdot A \left(x/ \eps \right) \grad w_\eps +k^2 n\left(x/ \eps \right) w_\eps=f_\epsilon \,  &\textrm{ in }& \,  D \label{homequ1} \\
 \Delta v_\eps + k^2 v_\eps=g_\epsilon  &\textrm{ in }& \,  D \label{homequ2} \\
w_\eps=v_\eps  &\textrm{ on }& \partial D\\
  \frac{\partial w_\eps}{\partial \nu_{A_\eps}}=\frac{\partial v_\eps}{\partial \nu} &\textrm{ on }& \partial D.\label{homequ3}
\end{eqnarray}
The following result is known (see  \cite{CCH01} and \cite{CH01wave} for the proof). 
\begin{lemma}
Assume that $A_{min}>1$ or $A_{max}<1$. Then the problem (\ref{homequ1})-(\ref{homequ3}) satisfies the Fredholm alternative. In particular it has a unique solution $(w_{\epsilon},v_{\epsilon})\in H^1(D)\times H^1(D)$ provided $k$ is not a transmission eigenvalue. 
\end{lemma}
The following lemma is proven in \cite{BCH} and \cite{cakginhad}.
\begin{lemma} 
Assume that $A_{min}>1$ or $A_{max}<1$ and either $n\equiv 1$ or if $n\not{\!\!\equiv}1$ then $\int_Y(n(y)-1)dy\neq 0$. The set of transmission eigenvalues $k\in{\mathbb C}$ is at most discrete with $+\infty$ as the only accumulation point.
\end{lemma}
Note that (\ref{estA}) implies that $A_h-I$ is positive definite if $A_{min}>1$ and $I-A_h$  is positive definite if $A_{max}<1$.

To analyze (\ref{homequ1})-(\ref{homequ3}) we introduce the variational space 
$$X(D):=\{ (w,v): \; w,v \in H^1(D) \, | \,  w-v \in H^1_0(D)\}$$
equipped with $H^1(D)\times H^1(D)$ norm and assume that $k$  is not a transmission eigenvalue for all $\epsilon>0$ small enough. Let $(w_\eps,v_\eps) \in  X(D)$ be the solution of  \eqref{homequ1}-\eqref{homequ3} for $\eps \geq 0$ small enough  (for $\eps=0$ we take the interior transmission problem with the homogenized coefficients $A_h$ and $n_h$) and assume that $(w_\eps,v_\eps)$ is a bounded sequence in $X(D)$-norm with respect to $\epsilon>0$ (this assumption will  be discussed later in the paper). This solution satisfies the variational problem
\begin{equation}
\int\limits_D A_\eps \grad w_\eps \cdot \grad {\varphi}_1 -\grad v_\eps \cdot \grad {\varphi}_2 -k^2(n_\eps w_\eps {\varphi}_1- v_\eps {\varphi}_2 )\, dx =\int\limits_D g_\epsilon { \varphi}_2\, dx-\int\limits_D f_\epsilon{\varphi}_1 \, dx \label{pa1}
\end{equation}
for all $(\varphi_1,\varphi_2) \in X(D)$. Hence we have that there is a $(w,v)
\in X(D)$ such that a subsequence $(w_\eps,v_\eps) \weakc (w,v)$  weakly in
$X(D)$ (strongly in $L^2(D) \times L^2(D)$). We now show that $(w,v)$ solves
the homogenized interior transmission  problem.  We adopt the formal 
 two-scale convergence framework: we say that a sequence $\alpha_\eps$ of $L^2(D)$
two-scale converges to $\alpha \in L^2(D\times Y)$ if 
$$
\int_D \alpha_\eps \varphi(x) \phi(x/\eps) dx \to \frac{1}{|Y|} \int_D \int_Y \alpha(x,y)
\varphi(x) \phi(y) dy dx
$$
for all $\varphi \in L^2(D)$ and $\phi \in C_\#(Y)$ (the space of $Y$-periodic continuous functions).  From \cite[Proposition 1.14]{allaire1} there exists
$w_1$ and $v_1$ $\in L^2(D, H^1_\#(Y))$ such that (up to a subsequence),
$\nabla w_\eps$ and $\nabla v_\eps$ respectively two-scale converge to
$\nabla_x w (x)+ \nabla_y w_1(x,y)$ and $\nabla_x v (x)+ \nabla_y
v_1(x,y)$. Let $\theta_1$ and $\theta_2$ in $C^\infty_0(D)$, $\phi_1$ and
$\phi_2$ in $C^\infty_\#(Y)$ ($Y$-periodic $C^\infty$ functions) and $(\psi_1,\psi_2) \in X(D)$. Applying
\eqref{pa1} to  $(\varphi_1,\varphi_2) \in X(D)$ such that $\varphi_i (x) =
\psi_i(x) + \eps \theta_i(x) \phi_i(x/\eps)$,
$i=1,2$ then taking the two-scale limit implies
\begin{multline} 
\int\limits_D \int\limits_Y  A(y) (\grad w(x) + \grad_y w_1(x,y)) \cdot (\grad
{\psi}_1(x) + \theta_1(x) \grad \phi_1(y))  dydx
\\
- \int\limits_D \int\limits_Y  (\grad v(x) + \grad_y v_1(x,y)) \cdot (\grad
{\psi}_2(x) + \theta_2(x) \grad \phi_2(y))  dydx
\\
- k^2 \int\limits_D \int\limits_Y n(y) w(x) {\psi}_1(x) - v(x) \psi_2(x) dydx 
=|Y| \int\limits_D g(x)  { \psi}_2(x)  - f(x){\psi}_1(x) \, dx.\label{phomtscale1}
\end{multline}
Taking $\psi_1=\psi_2=0$ one easily deduces 
\begin{equation} \label{corrector1} 
w_1(x,y) = - \vec{\psi}(y) \cdot \grad w(x) + \overline{w}_1 (x) \mbox{ and }
v_1(x,y) = \overline{v}_1 (x).
\end{equation}
Then considering again \eqref{phomtscale1} with $\theta_1 = \theta_2 =0$ implies
that $(w,v) \in X(D)$ satisfies
\begin{equation}
\int\limits_D A_h \grad w \cdot \grad {\psi}_1 -\grad v_\eps \cdot \grad
{\psi}_2 -k^2(n_h w {\psi}_1- v {\psi}_2 )\, dx =\int\limits_D g { \psi}_2\,
dx-\int\limits_D f {\psi}_1 \, dx \label{homtev1}
\end{equation}
which is the variational formulation of the homogenized problem \eqref{h1}-\eqref{h2}.

The above analysis was based on the assumption that the sequence that solves \eqref{homequ1}-\eqref{homequ3} is bounded with respect to $\epsilon>0$. Now we wish to show that any sequence that solves \eqref{homequ1}-\eqref{homequ3} is indeed bounded independently of $\eps$. 
\begin{theorem} \label{imthm1} Assume that either $A_{min}>1$ or $A_{max}<1$ and  that $k$ is not a transmission eigenvalue for $\eps \geq 0$ small enough.  Then for any $(w_\eps,v_\eps)$ solving  \eqref{homequ1}-\eqref{homequ3} there exists  $C>0$ independent of $(f_\eps,g_\eps)$ and $\eps$ such that
$$||w_\eps||_{H^1(D)} +||v_\eps||_{H^1(D)} \leq C \left(||f_\eps||_{L^2(D)} +||g_\eps||_{L^2(D)} \right).$$
\end{theorem}
\begin{proof}
We will prove the Fredholm property following the $\mathbb{T}$-coercivity approach in \cite{BCH}.  To this end we recall the variational formulation  (\ref{pa1}) equivalent to \eqref{homequ1}-\eqref{homequ3}. Let us first assume that $A_{min}>1$, which means that $A_{\epsilon}-I$ is positive definite in $D$ uniformly with respect to $\epsilon>0$, and define the bounded sesquilinear forms in $X(D)\times X(D)$ 
\begin{eqnarray*}
a_\eps \big( (w_\eps,v_\eps) ; (\varphi_1,\varphi_2) \big)&\hspace*{-0.2cm}:=\hspace*{-0.2cm}&\int\limits_{D} A_\eps \grad w_\eps \cdot \grad \overline{\varphi}_1 + A_{min} w_\eps \overline{\varphi}_1 \, dx -\int\limits_{D}   \grad v_\eps \cdot \grad \overline{\varphi}_2 + v_\eps \overline{\varphi}_2 \, dx,\\
b_\eps \big( (w_\eps,v_\eps) ; (\varphi_1,\varphi_2) \big)&\hspace*{-0.2cm}:=\hspace*{-0.2cm}&-\int\limits_{D} (k^2 n_\eps+A_{min}) w_\eps \overline{\varphi}_1 \, -(k^2+1) v_\eps \overline{\varphi}_2 \, dx.
\end{eqnarray*}
Then  (\ref{pa1})  can be written as 
$$a_\eps \big( (w_\eps,v_\eps) ; (\varphi_1,\varphi_2) \big)+b_\eps \big( (w_\eps,v_\eps) ; (\varphi_1,\varphi_2) \big)=F_\eps(\varphi_1,\varphi_2)$$
where $F_\eps(\varphi_1,\varphi_2)$ is the bounded linear functional on $X(D)$ defined by the right hand side of (\ref{pa1}). Let us consider ${\mathbb A}_{\epsilon}:X(D)\to X(D)$ and ${\mathbb B}_{\epsilon}:X(D)\to X(D)$ the bounded linear operators defined from $a_\eps \big( (w_\eps,v_\eps) ; (\varphi_1,\varphi_2) \big)$ and $b_\eps \big( (w_\eps,v_\eps) ; (\varphi_1,\varphi_2) \big)$  by  means of  Riesz representation theorem. It is clear that ${\mathbb B}_{\epsilon}$ is compact. We next show that  ${\mathbb A}_{\epsilon}$ is invertible with bounded inverse uniformly with respect to $\epsilon> 0$. To this end we consider the isomorphism $\mathbb{T}(w,v)=(w, -v+2w): X(D) \mapsto X(D)$ (it is easy to check that $\mathbb{T}=\mathbb{T}^{-1}$) and show that $a_\eps \big( (w_\eps,v_\eps) ; \mathbb{T} (\varphi_1,\varphi_2) \big)$ is coercive in $ X(D)$. Note that the isomorphism ${\mathbb T}$ does not depend on $\epsilon$. Hence, we have that
\begin{eqnarray*}
\left| a_\eps \big( (w_\eps,v_\eps) ; \mathbb{T}(w_\eps,v_\eps) \big) \right| &\geq& \int\limits_{D} A_\eps \grad w_\eps \cdot \grad \overline{w}_\eps + A_{min} |w_\eps|^2 \, dx +\int\limits_{D}  | \grad v_\eps|^2 +|v_\eps|^2 \, dx\\
&-&2 \left| \,  \int\limits_{D}  \grad v_\eps \cdot \grad \overline{w}_\eps +  v_\eps \overline{w}_\eps \, dx \right|. 
\end{eqnarray*}
But we can estimate
$$\left| 2\int\limits_{D}  \grad v_\eps \cdot \grad \overline{w}_\eps +  v_\eps \overline{w}_\eps \, dx \right| \leq  \frac{1}{\delta}||w_\eps||^2_{H^1(D)} +\delta  ||v_\eps||^2_{H^1(D)} \quad \text{ for any } \, \, \, \delta>0.$$ 
Hence we obtain 
\begin{eqnarray*}
\left| a_\eps \big( (w_\eps,v_\eps) ; \mathbb{T}(w_\eps,v_\eps) \big) \right| &\geq& \left(A_{min}- \frac{1}{\delta} \right) ||w_\eps||^2_{H^1(D)} +(1-\delta) ||v_\eps||^2_{H^1(D)}.
\end{eqnarray*}
So for any $\delta \in \left( \frac{1}{A_{min}}, 1\right)$ we have that there is a constant $\alpha >0$ independent of $\eps$ such that
$$\left| a_\eps \big( (w_\eps,v_\eps) ; \mathbb{T}(w_\eps,v_\eps) \big) \right| \geq \alpha \left( ||w_\eps||^2_{H^1(D)} +||v_\eps||^2_{H^1(D)} \right).$$
Next we assume that $A_{max} <1$ which means that  $I-A_{\epsilon}$ is positive definite in $D$ uniformly with respect to $\epsilon>0$. Similarly  we define
\begin{eqnarray*}
a_\eps \big( (w_\eps,v_\eps) ; (\varphi_1,\varphi_2) \big)&\hspace*{-0.2cm}:=\hspace*{-0.2cm}&\int\limits_{D} A _\eps\grad w_\eps \cdot \grad \overline{\varphi_1} + A_{max} w_\eps \overline{\varphi_1} \, dx -\int\limits_{D}   \grad v_\eps \cdot \grad \overline{\varphi_2} + v_\eps \overline{\varphi_2} \, dx\\
b_\eps \big( (w_\eps,v_\eps) ; (\varphi_1,\varphi_2) \big)&\hspace*{-0.2cm}:=\hspace*{-0.2cm}&-\int\limits_{D} (k^2 n_\eps+A_{max}) w_\eps \overline{\varphi_1} \, -(k^2+1) v_\eps \overline{\varphi_2} \, dx
\end{eqnarray*}
and the corresponding bounded linear operator ${\mathbb A}_{\epsilon}:X(D)\to X(D)$ and ${\mathbb B}_{\epsilon}:X(D)\to X(D)$. To show that ${\mathbb A}_{\epsilon}$  is invertible we now consider the isomorphism $\mathbb{T}(w,v)=(w-2v, -v): X(D) \mapsto X(D)$ (again it is easy to check that $\mathbb{T}=\mathbb{T}^{-1}$). We then have that 
\begin{eqnarray*}
\left| a_\eps \big( (w_\eps,v_\eps) ; \mathbb{T}(w,v) \big) \right| &\geq& \int\limits_{D} A_\eps \grad w_\eps \cdot \grad \overline{w}_\eps + A_{max} |w_\eps|^2 \, dx +\int\limits_{D}  | \grad v_\eps|^2 +|v_\eps|^2 \, dx\\
&-&2 \left| \,  \int\limits_{D}  A_\eps \grad w_\eps \cdot \grad \overline{v}_\eps +  A_{max} w_\eps \overline{v}_\eps \, dx \right|.
\end{eqnarray*}
Using that $A_\eps$ is symmetric positive definite we have that for any $\delta>0$: 
$$\left| 2  \int\limits_{D} A_\eps \grad w_\eps \cdot \grad \overline{v}_\eps \, dx \right| \leq  \delta \int\limits_{D} A_\eps \grad w_\eps \cdot \grad \overline{w}_\eps \, dx +\frac{A_{max}}{\delta} \int\limits_{D}  |\grad  v|_\eps^2\, dx$$
We also use that for any $\mu>0$:
$$ \left| 2 \int\limits_{D} A_{max} w_\eps \overline{v}_\eps \, dx \right| \leq \frac{A^2_{max}}{\mu} ||w_\eps||^2_{L^2(D)} +\mu  ||v_\eps||^2_{L^2(D)}$$
From the above inequalities we see that: 
\begin{eqnarray*}
\left| a_\eps \big( (w_\eps,v_\eps) ; \mathbb{T}(w_\eps,v_\eps) \big) \right| &\geq&A_{min} \left(1- \delta \right) ||\grad w_\eps||^2_{L^2(D)} +\left(1-\frac{A_{max}}{\delta} \right) ||\grad v_\eps||^2_{L^2(D)} \\
&+&A_{max} \left(1-\frac{A_{max}}{\mu} \right)||w_\eps||^2_{L^2(D)} + (1-\mu) ||v_\eps||^2_{L^2(D)}
\end{eqnarray*}
 for any $\mu \, ,\delta \in (A_{max} , 1)$. Hence $ \mathbb{A}^{-1}_\eps: X(D) \mapsto X(D)$ exists for all $\eps >0$ with $|| \mathbb{A}_\eps^{-1} ||_{\mathcal{L}(X(D))}$ bounded independently of $\eps$. The above analysis also proves that  the Fredholm alternative can be applied to  the operator $({\mathbb A}_\epsilon+{\mathbb B}_\epsilon)$ and equivalently to  \eqref{homequ1}-\eqref{homequ3}. Therefore if  $k$ is not a transmission eigenvalue for $\eps \geq 0$ we have that there is a constant $C_\eps$ that does not depend on $(f_\eps,g_\eps)$ but possibly on $\epsilon>0$ such that the unique solution $(w_\eps,v_\eps)$ of  \eqref{homequ1}-\eqref{homequ3}
$$||w_\eps||_{H^1(D)} +||v_\eps||_{H^1(D)} \leq C_\eps \left(||f_\eps||_{L^2(D)} +||g_\eps||_{L^2(D)} \right).$$
The above analysis show that if $(w_\epsilon,v_\epsilon)\in X(D)$ solves \eqref{homequ1}-\eqref{homequ3} then 
$$({\mathbb I}+{\mathbb K}_\epsilon)(w_\epsilon, v_\epsilon)=(\alpha_\epsilon, \beta_\epsilon)$$ where ${\mathbb K}_\epsilon$ is compact such that
\begin{equation}\label{eq1}
|| \mathbb{K}_\eps (w_\eps ,v_\eps)||_{X(D)} \leq M_1 \left( ||w_\eps||_{L^2(D)} + ||v_\eps||_{L^2(D)} \right)
\end{equation}
and $(\alpha_\eps, \beta_\eps) \in X(D)$ is such that
\begin{equation}\label{eq2}
 ||\alpha_\eps||_{H^1(D)} + ||\beta_\eps||_{H^1(D)} \leq M_2  \left( ||f_\eps||_{L^2(D)} + ||g_\eps||_{L^2(D)} \right)
\end{equation}
with $M_1$ and $M_2$ independent of $\epsilon$ (Note that (\ref{eq1}) holds for ${\mathbb K}={\mathbb A}^{-1}_\epsilon{\mathbb B}_\epsilon$ since obviously $\|{\mathbb B}_{\epsilon}(w_\eps ,v_\eps)\|_{X(D)}$ is bounded by the $L^2(D)\times L^2(D)$ norm of $(w_\eps ,v_\eps)$ and $\|\mathbb{A}^{-1}_\epsilon\|$ is uniformly bounded with respect to $\epsilon$).\\
Next we need to show that  $C_\eps$ is bounded independently of $\eps$. Assume to the contrary that $C_\eps$ is not bounded as $\epsilon\to 0$. If this is true we can find a subsequence such that
$$||w_\eps||_{L^2(D)} +||v_\eps||_{L^2(D)} \geq \gamma_\eps \left(||f_\eps||_{L^2(D)} +||g_\eps||_{L^2(D)} \right)$$
where the sequence $\gamma_\eps \overset{\tiny{ \eps \rightarrow 0} }{\longrightarrow} \infty$. So we define the sequence $(\tilde{w}_\eps , \tilde{v}_\eps) \in X(D)$
$$\tilde{w}_\eps := \frac{ w_\eps}{|| w_\eps ||_{L^2(D)} + || v_\eps ||_{L^2(D)}  }  \, \, \,  \textrm{ and } \, \, \,  \tilde{v}_\eps := \frac{ v_\eps}{|| w_\eps ||_{L^2(D)} + || v_\eps ||_{L^2(D)}  }. $$
Notice that $|| \tilde{w}_\eps ||_{L^2(D)} + || \tilde{v}_\eps ||_{L^2(D)} =1 $ and $(\tilde{w}_\eps , \tilde{v}_\eps)$ solves \eqref{homequ1}-\eqref{homequ3} with $( \tilde{f}_\eps , \tilde{g}_\eps) \in L^2(D) \times L^2(D)$ given by
$$\tilde{f}_\eps := \frac{ f_\eps}{|| w_\eps ||_{L^2(D)} + || v_\eps ||_{L^2(D)}  }  \, \, \,  \textrm{ and } \, \, \, \tilde{g}_\eps := \frac{ g_\eps}{|| w_\eps ||_{L^2(D)} + || v_\eps ||_{L^2(D)}  }. $$
Furthermore we have that $|| \tilde{f}_\eps ||_{L^2(D)} + || \tilde{g}_\eps ||_{L^2(D)} \leq \frac{1}{\gamma_\eps}  \overset{\tiny{ \eps \rightarrow 0} }{\longrightarrow} 0 $ and $(\mathbb{I}+\mathbb{K}_\eps)(\tilde{w}_\eps , \tilde{v}_\eps)= (\tilde{\alpha}_\eps, \tilde{\beta}_\eps)$,  where $\tilde{\alpha}_\eps, \tilde{\beta}_\eps$ are defined from $\tilde{f}_\epsilon$ and $\tilde{g}_\epsilon$ as above. Now from (\ref{eq1}) and (\ref{eq2})  we have that for all $\eps$ sufficiently small 
\begin{align*}
&|| \tilde{w}_\eps ||_{H^1(D)} + || \tilde{v}_\eps ||_{H^1(D)}  \leq  ||\mathbb{K}_\eps(\tilde{w}_\eps , \tilde{v}_\eps)||_{X(D)}+ ||(\tilde{\alpha}_\eps, \tilde{\beta}_\eps)||_{X(D)}, \\
									&\hspace{2cm}\leq M_1 \left( || \tilde{w}_\eps ||_{L^2(D)} + || \tilde{v}_\eps ||_{L^2(D)} \right) +M_2 \left(|| \tilde{f}_\eps ||_{L^2(D)} + || \tilde{g}_\eps ||_{L^2(D)} \right) ,\\
									&\hspace{2cm} \leq M_1+M_2.
\end{align*}
Since $M_1$ and $M_2$ are independent of $\eps$ we have that $(\tilde{w}_\eps , \tilde{v}_\eps)$ is a bounded sequence in $X(D)$ and therefore has a subsequence that converges to $(\tilde{w}, \tilde{v})$ weakly in $X(D)$ (strongly in $L^2(D) \times L^2(D)$). Also we have that $(\tilde{w}, \tilde{v})$ solves \eqref{homtev1} with $(f,g)=(0,0)$. Since $k$ is not a transmission eigenvalue for $\eps=0$ we have that $(\tilde{w}, \tilde{v})=(0,0)$ which contradicts the fact that $|| \tilde{w} ||_{L^2(D)} + || \tilde{v} ||_{L^2(D)} =1$ which proves the claim. 
\end{proof}

Notice that Theorem \ref{imthm1} gives that any sequence $(w_\eps, v_\eps)$ that solves \eqref{homequ1}-\eqref{homequ3} is bounded in $X(D)$ since $f_\eps$ and $g_\eps$ are assumed to converge strongly in $L^2(D)$. We can now state the following convergence result given by the above analysis. 

\begin{theorem} Assume that either $A_{min}>1$ or $A_{max}<1$ and  that $k$
  is not a transmission eigenvalue for $\eps \geq 0$ small enough.  Then we
  have that $(w_\eps, v_\eps)$ solving  \eqref{homequ1}-\eqref{homequ3}
  converges weakly in $X(D)$ \big(strongly in $L^2(D) \times L^2(D)$\big) to
  $(w,v)$ that is a solution of \eqref{homtev1}. If we assume in addition that
  $w\in H^2(D)$ then, $v_\eps$ strongly
  converges to $v$ in $H^1(D)$ and  $w_\eps(x)-  w(x) - 
  \eps w_1(x, x/\eps)$ strongly converges to $0$ in $H^1(D)$ where 
$w_1(x,y) := -\vec \psi(y)\cdot \nabla w(x)$.
\end{theorem}
\begin{proof}
The first part of the theorem is a direct consequence of the above analysis and
the uniqueness of solutions to \eqref{homtev1}. The corrector type result 
is
obtained using the T-coercivity property as follows. 
We first observe that, due to the strong convergence of the right hand side of
the variational formulation of interior transmission problem, we have that 
$$
(a_\eps + b_\eps) \big( (w_\eps,v_\eps) ; \mathbb{T}(w_\eps,v_\eps) \big) 
\rightarrow F(\mathbb{T}(w,v)) =
(a+b)\big( (w,v) ; \mathbb{T}(w,v) \big) 
$$
as $\eps \rightarrow 0$ where $a$ and $b$ have similar expressions as $a_\eps$ and
$b_\eps$ with $A_\eps$ and $n_\eps$ respectively replaced by $A_h$ and
$n_h$ and $F$ has the same expression as $F_\eps$ with $f_\eps$ and $g_\eps$
respectively replaced with $f$ and $g$. The $L^2$ strong convergence implies that   
$$b_\eps \big(
(w_\eps,v_\eps) ; \mathbb{T}(w_\eps,v_\eps) \big) \rightarrow 
 b \big(
(w,v) ; \mathbb{T}(w,v) \big). $$
We therefore end up with, 
\begin{equation}\label{trick1}
a_\eps \big( (w_\eps,v_\eps) ; \mathbb{T}(w_\eps,v_\eps) \big) \rightarrow 
a\big( (w,v) ; \mathbb{T}(w,v) \big)  
\end{equation}
as $\eps \rightarrow 0$. Let us set  $w_1^\eps(x) := w_1(x, x/\eps)$. From the expression of  $w_1$ one has (see
for instance \cite{shari1}) 
$$
\eps^{1/2} \|w_1^\eps\|_{H^{1/2}(\partial D)} \le C 
$$
for some constant $C$ independent of $\eps$. Therefore we can construct a lifting
function $v_1^\eps \in H^1(D)$ such that $v_1^\eps = w_1^\eps$ on $\partial D$
and 
\begin{equation} \label{nextric}
\eps \|v_1^\eps\|_{H^1(D)} \to 0 \mbox{ as } \eps \to 0.
\end{equation}
Now, taking as test functions $\varphi_1
= \tilde w_\eps $ and $\varphi_2 = \tilde v_\eps$ where $\tilde w_\eps(x):= w(x) + 
  \eps w_1(x, x/\eps)$ and $\tilde v_\eps(x):= v(x) + 
  \eps v_1^\eps(x)$, one has
$$
(a_\eps + b_\eps) \big( (w_\eps,v_\eps) ; \mathbb{T}(\tilde w_\eps,\tilde v_\eps) \big) 
\rightarrow F(\mathbb{T}(w,v)).
$$
Using the two-scale convergence of the
  sequences $w_\eps$ and $v_\eps$ together with the form (and regularity) of
  $w_1$ as well as \eqref{nextric}, we easily see that
$$
b_\eps \big( (w_\eps,v_\eps) ; \mathbb{T}(\tilde w_\eps,\tilde v_\eps) \big)
\rightarrow 
 b \big(
(w,v) ; \mathbb{T}(w,v) \big)
$$
while 
$$
 a_\eps \big( (w_\eps,v_\eps) ; \mathbb{T}(\tilde w_\eps,\tilde v_\eps) \big)
\rightarrow L(w,w_1,v)
$$
with 
\begin{multline*}
 L(w,w_1,v) = \frac{1}{|Y|}
\int\limits_D \int\limits_Y A(y) (\nabla w(x)+ \nabla_y w_1(x,y)) \cdot
(\nabla\overline w(x)+ \nabla_y \overline w_1(x,y))  dy dx 
\\
+ \int\limits_D |\nabla v(x)|^2 +
A_{min} |w(x)|^2 + |v(x)|^2 - 2   \nabla \overline w(x)
\nabla v(x) -2 \overline w(x) v(x)  dx
\end{multline*}
 in the case $A_{min} >1$ and 
\begin{multline*}
 L(w,w_1,v) = \frac{1}{|Y|}
\int\limits_D \int\limits_Y A(y) (\nabla w(x)+ \nabla_y w_1(x,y)) \cdot
(\nabla\overline w(x)+ \nabla_y \overline w_1(x,y))  dy dx 
\\
-2 \frac{1}{|Y|}
\int\limits_D \int\limits_Y A(y) (\nabla w(x)+ \nabla_y w_1(x,y)) \cdot
\nabla\overline v(x)  dy dx 
\\
+ \int\limits_D |\nabla v(x)|^2 +
A_{min} |w(x)|^2 + |v(x)|^2  -2 A_{\min}\overline v(x) w(x)   dx
\end{multline*}
 in the case $A_{max} < 1$.  Hence we can conclude that 
$$
F(\mathbb{T}(w,v)) = L(w,w_1,v) + b \big(
(w,v) ; \mathbb{T}(w,v) \big)
$$
and therefore
\begin{equation} \label{trick2}
a\big( (w,v) ; \mathbb{T}(w,v) \big)  = L(w,w_1,v).
\end{equation}
Using \eqref{trick1} and \eqref{trick2} and the T-coercivity, we can apply similar arguments as
in \cite[Theorem 2.6]{allaire1} to obtain the result. Indeed, the T-coercivity shows that it is sufficient to prove
that 
\begin{equation}\label{trick3}
a_\eps \big( (w_\eps-\tilde w_\eps,v_\eps-v) ; \mathbb{T}(w_\eps-\tilde
w_\eps,v_\eps-v) \big) \to 0.
\end{equation}
Now, using the two-scale convergence of the sequences $v_\eps$ and $w_\eps$, we 
observe that each of the quantities 
$$
\begin{array}{l}
a_\eps \big( (w_\eps,v_\eps) ; \mathbb{T}(\tilde w_\eps,v) \big) , \; 
a_\eps \big( (\tilde w_\eps,v) ; \mathbb{T}(w_\eps,v_\eps) \big) 
\mbox{ and } a_\eps \big( (\tilde w_\eps,v) ; \mathbb{T}(\tilde w_\eps,v) \big) 
\end{array}
$$
converges to $L(w, w_1, v)$

Finally, using \eqref{trick1} we can conclude that
$$
a_\eps \big( (w_\eps-\tilde w_\eps,v_\eps-v) ; \mathbb{T}(w_\eps-\tilde
w_\eps,v_\eps-v) \big) \to a\big( (w,v) ; \mathbb{T}(w,v) \big) - L(w, w_1, v)
$$
and then the result is a direct consequence of \eqref{trick3}. 
\end{proof}

\subsection{The case of $A_\eps \equiv I$} \label{sec3.2} Here we now  assume that either $n_{min}>1$ or $0<n_{max}<1$. For the case where $A_\eps \equiv I$ the interior transmission problem becomes: Find $(w_\eps , v_\eps )\in L^2(D) \times L^2(D)$ such that
\begin{eqnarray}
\Delta w_\eps +k^2 n\left(x/ \eps \right) w_\eps=0 \,  &\textrm{ in }& \,  D\label{aeq11} \\
 \Delta v_\eps + k^2 v_\eps=0  &\textrm{ in }& \,  D \\
 w_\eps-v_\eps=f_\eps  &\textrm{ on }& \partial D \label{hap1}\\
 \frac{\partial w_\eps}{\partial \nu}-\frac{\partial v_\eps}{\partial \nu}=g_\eps &\textrm{ on }& \partial D \label{aeq12}
\end{eqnarray}
for the boundary data  $(f_\eps , g_\eps) \in H^{3/2}(\partial D) \times H^{1/2}(\partial D)$ converging strongly  to $(f,g)\in H^{3/2}(\partial D) \times H^{1/2}(\partial D)$ as $\epsilon \to 0$. Just as in the case for anisotropic media we require that $k^2$ is not a transmission eigenvalue for $\eps \geq 0$ small enough. We formulate the interior transmission problem for the difference $U_\eps :=w_\eps -v_\eps \in H^2(D)$. Using the interior transmission problem one can show that  this $U_\epsilon$ satisfies 
\begin{eqnarray}
0  &=& \left(\Delta +k^2  n_\eps \right) \frac{1}{n_\eps-1} \left(\Delta +k^2  \right) U_\eps \, \, \textrm{ in } \,  D \label{4thitp}
\end{eqnarray}
where
\begin{eqnarray}
v_\eps &=& - \frac{1}{k^2 (n_\eps-1)} \left( \Delta U_\eps +k^2   n_\eps U_\eps \right) \, \, \textrm{ in } \,  D \label{dv}\\
w_\eps &=& - \frac{1}{k^2 (n_\eps-1)} \left( \Delta U_\eps +k^2  U_\eps \right) \, \, \textrm{ in } \,  D \label{dw}
\end{eqnarray}
\begin{theorem}\label{th3.3}
Assume that either $(n_{min}-1)>0$ or $(n_{max}-1)<0$ and $U_\eps \in H^2(D)$ is a bounded sequence, then there is a subsequence such that $U_\eps \weakc U$ in $H^2(D)$ and $(w_\eps , v_\eps) \weakc (w,v)$ in $L^2(D) \times L^2(D)$ \big(strongly in $L^2_{loc}(D) \times L^2_{loc}(D)$\big). Moreover we have that the limit $U$ satisfies 
\begin{eqnarray}
 \left(\Delta +k^2  n_h \right) \frac{1}{n_h-1} \left(\Delta +k^2  \right) U=0 \, \, &\textrm{ in }& \,  D,\label{39} \\
  U=f  \;\;\;\textrm{ and }  \;\;\; \frac{\partial U}{\partial \nu}=g &\textrm{ on }& \partial D,
 \end{eqnarray}
$U=w-v$, and $(w,v)$ satisfy
\begin{eqnarray}
 \Delta v +k^2 v=0 \;\;\; &\textrm{ and }& \;\;\; \Delta w +k^2 n_h w=0\;\; \;\; \textrm{ in } \;\;  D, \label{hap2}\\
 w-v=f\;\;\;  &\textrm{ and }&  \;\;\; \frac{\partial w}{\partial \nu}-\frac{\partial v}{\partial \nu}=g \;\; \;\; \;\; \;\; \;\; \textrm{ on }\;\; \partial D. \label{hap3}
 \end{eqnarray}
\end{theorem}
\begin{proof}
Since $U_\eps$ is a bounded sequence in $H^2(D)$, from (\ref{dv}) and (\ref{dw}) we have that $(w_\eps , v_\eps)$ is a bounded sequence in $L^2(D) \times L^2(D)$. Therefore we have that there is a subsequence still denoted by $(w_\eps , v_\eps)$ that is weakly convergent in $L^2(D) \times L^2(D)$. So we have that for all $\varphi \in \mathcal{C}^{\infty}_0(D)$, there is a $v \in L^2(D)$ such that:   
$$ 0=\int\limits_D v_\eps (\Delta \varphi + k^2 \varphi ) \, dx  \overset{\tiny{ \eps \rightarrow 0} }{\longrightarrow}  \int\limits_D v (\Delta \varphi +k^2 \varphi ) \, dx.$$
This gives that $\Delta v + k^2 v=0$ in the distributional sense. By interior elliptic regularity (see e.g. \cite{wloka}) for all $\Omega \subset \overline{\Omega} \subset D$ and all  $\eps > 0$ we have 
$$ || v_\eps ||_{H^1(\Omega)} \leq C $$
for some constant  independent of $\eps$ which implies (using an increasing sequence of
domains $\Omega_n$ that converges to $D$ and a diagonal extraction process of
the subsequence)  that  a subsequence $v_\eps$  converges to $v$ strongly  in $L^2_{loc}(D)$.
Next  since $w_\eps=U_\eps + v_\eps$ and $U_\epsilon$ is bounded in $H^2(D)$, we  have that $w_\eps$ converges  to some $w$ weakly in $L^2(D)$ and strongly in $L^2_{loc}(D)$. Now using the strong convergence we have that for all $\varphi \in \mathcal{C}^{\infty}_0(D)$ such that  $\overline{\text{supp}(\varphi)} \subset D$ we obtain that 
$$ 0=\int\limits_D w_\eps (\Delta \varphi + k^2 n_\eps \varphi ) \, dx  \overset{\tiny{ \eps \rightarrow 0} }{\longrightarrow}  \int\limits_D w (\Delta \varphi + k^2 n_h \varphi ) \, dx,$$
which gives that $\Delta w + k^2 n_h w=0$ in the distributional sense. Now, the
fact that $-k^2 (n_\eps-1) w_\eps =  \Delta U_\eps +k^2 U_\eps$, the weak
convergence of  $U_\epsilon$ to $U$ in $H^2(D)$ and the local strong convergence
of $w_\epsilon$ to the above $w$ imply that the limit $U$  satisfies
$\left(\Delta +k^2  n_h \right) \frac{1}{n_h-1} \left(\Delta +k^2  \right) U=0$
in $D$ and $U=w-v$. Finally, integration by parts formulas  together with
(\ref{hap1}) and  (\ref{aeq12}) guaranty that $U:=w-v$ satisfies the boundary conditions (\ref{hap2}) and (\ref{hap3}) which ends the proof.
\end{proof}

The above result that connects $w_\epsilon$, $v_\epsilon$ and $U_\epsilon$ with  the respective limits  requires that $U_\epsilon$ is a bounded sequence. Next we show that this is the case for every solution to the interior transmission problem.
To this end, since $(f_\eps , g_\eps) \in H^{3/2}(\partial D) \times H^{1/2}(\partial D)$ there is a lifting function $\phi_\eps \in  H^2(D)$ such that $\phi_\eps \big|_{\partial D}=f_\epsilon$ and $\frac{\partial \phi_\eps}{\partial \nu} \big|_{\partial D}=g_\eps$ and 
\begin{equation}\label{est}
 ||\phi_\eps||_{H^2(D)} \leq C \left( ||f_\eps||_{H^{3/2}(\partial D)} +||g_\eps||_{H^{1/2}(\partial D)}\right)
\end{equation} 
where the constant $C$ is independent of $\eps$ and 
$\phi_\eps \to \phi $
strongly in
$H^2(D)$  where  $\phi \big|_{\partial D}=f$ and $\frac{\partial \phi}{\partial \nu} \big|_{\partial D}=g$. Now following \cite{cakginhad} and \cite{cakhad} we  define the bounded sesquilinear forms on  $H^2_0(D) \times H^2_0(D)$:
\begin{eqnarray}
\mathcal{A}_\eps(u,\varphi) &=&  \int\limits_D  \frac{1}{ n_\eps-1} \left[ \left( \Delta u +k^2 u \right) \left( \Delta \overline{\varphi} +k^2 \overline{\varphi} \right)\right] +k^4 u \overline{\varphi}  \, dx, \label{bad1}\\
\widehat{\mathcal{A}}_\eps(u,\varphi) &=& \int\limits_D  \frac{n_\eps}{ 1-n_\eps} \left[ \left( \Delta u +k^2 u \right) \left( \Delta \overline{\varphi} +k^2 \overline{\varphi} \right)\right] +\Delta u \Delta \overline{\varphi}  \, dx, \label{bad2}\\
\mathcal{B}(u,\varphi) &=& \int\limits_D \grad u \grad \overline{\varphi}  \, dx. \label{bad3}
\end{eqnarray}
With the help of the lifting function $\phi_\epsilon$, we have that  $u_\eps \in H^2_0(D)$  where $U_\eps=u_\eps+\phi_\eps$ and  that $u_\eps$ solve the variational problems
\begin{eqnarray}
\mathcal{A}_\eps(u_\eps,\varphi) -k^2\mathcal{B}(u_\eps,\varphi) =  L_\eps (\varphi) \label{h2var1} \\
\widehat{\mathcal{A}}_\eps(u_\eps,\varphi) -k^2\mathcal{B}(u_\eps,\varphi) =   \widehat{L}_\eps (\varphi) \label{h2var2} 
\end{eqnarray}
where the conjugate linear functionals  are defined as follows
$$L_\eps (\varphi) =k^2\mathcal{B}(\phi_\eps,\varphi)-\mathcal{A}_\eps(\phi_\eps,\varphi) \quad \text{ and }\quad  \widehat{L}_\eps (\varphi)=k^2\mathcal{B}(\phi_\eps,\varphi)-\widehat{\mathcal{A}}_\eps(\phi_\eps,\varphi).$$
Let ${\mathbb A}_\epsilon: H^2_0(D)\to H^2_0(D)$,  $\widehat{{\mathbb A}}_\epsilon: H^2_0(D)\to H^2_0(D)$ and ${\mathbb B}: H^2_0(D)\to H^2_0(D)$  be bounded linear operators defined by the sesquilinear forms (\ref{bad1}), (\ref{bad2}) and (\ref{bad3}) by means of Riesz representation theorem. Obviously ${\mathbb B}$ is a compact operator and it does not depend on $\epsilon$, and furthermore $\|{\mathbb B}(u_\epsilon)\|_{H^2(D)}$ is bounded by $\|u_\epsilon\|_{H^1(D)}$.   
In \cite{cakhad}  it is shown that $\mathcal{A}_\eps(\cdot, \cdot)$ is coercive when $\frac{1}{ n_\eps-1}  \geq \alpha >0$ for all $\eps >0$ (which is satisfied if $n_{min}>1$)  whereas  $\widehat{\mathcal{A}}_\eps(\cdot, \cdot)$ is coercive  when $ \frac{n_\eps}{ 1-n_\eps }\geq \alpha >0$ for all $\epsilon>0$ (which is satisfied if $0<n_{max}<1$) and furthermore the coercivity constant depends only on $D$ and $\alpha$. Hence ${\mathbb A}_\epsilon^{-1}$ exists if  $n_{min}>1$ and $\widehat{{\mathbb A}}^{-1}_\epsilon$ exists if $0<n_{max}<1$ and their norm is uniformly bounded with respect to $\epsilon$.  
\begin{theorem}\label{th3.4}
Assume that either $n_{min}>1$ or $0<n_{max}<1$, and that $k$ is not a transmission eigenvalue for $\epsilon\geq 0$ small enough. If $U_\eps \in H^2(D)$ is a solution to \eqref{4thitp} such that  $U_\eps=f_\eps  \textrm{ and }  \frac{\partial U_\eps}{\partial \nu}=g_\eps \textrm{ on } \partial D$,  then there is a constant $C>0$ independent of $\epsilon\geq 0$ and $(f_\eps,g_\eps)$ such that: 
$$||U_\eps||_{H^2(D)} \leq C \left(  ||f_\eps||_{H^{3/2}(\partial D)} +||g_\eps||_{H^{1/2}(\partial D)}  \right).$$
\end{theorem} 

\begin{proof}
First recall that  $U_\eps=u_\eps+\phi_\eps$ where $u_\eps \in H^2_0(D)$ satisfies either (\ref{h2var1}) or (\ref{h2var2}) and $\phi_\eps \in  H^2(D)$ satisfies (\ref{est}). Therefore it is sufficient to prove the result for $u_\eps$. From the discussion above we know that $u_\epsilon$ satisfies
\begin{equation}\label{fred}
({\mathbb I}-k^2{\mathbb K}_\epsilon)(u_\epsilon)=\alpha_\epsilon
\end{equation}
where ${\mathbb K}_\epsilon={\mathbb A}_\epsilon^{-1}{\mathbb B}$ and $\alpha_\epsilon\in H^2_0(D)$ is the Riesz representation of $L_\eps$ if $n_{min}>1$, and  ${\mathbb K}_\epsilon=\widehat{{\mathbb A}}_\epsilon^{-1}{\mathbb B}$ and $\alpha_\epsilon\in H^2_0(D)$ is the Riesz representation of $\widehat{L}_\eps$ if $0<n_{max}<1$. In both cases
$$
\|{\mathbb K}_\epsilon(u_\epsilon)\|_{H^2(D)}\leq M_1\|u_\epsilon\|_{H^1(D)} 
$$
and
$$
\|\alpha_\epsilon\|_{H^2(D)}\leq M_2 \left(||f_\eps||_{H^{3/2}(\partial D)} +||g_\eps||_{H^{1/2}(\partial D)} \right)
$$
with $M_1$ and $M_2$ independent of $\epsilon>0$.  Now since $k^2$ is not a
transmission eigenvalue for $\epsilon\geq 0$ (small enough), the Fredholm alternative applied to (\ref{fred}) guaranties the existence of a constant  $C_\eps$ independent of $f_\epsilon,g_\epsilon$ such that $$||u_\eps||_{H^2(D)} \leq C_\eps \left(||f_\eps||_{H^{3/2}(\partial D)} +||g_\eps||_{H^{1/2}(\partial D)} \right).$$ In the same way as in Theorem \ref{imthm1}, we can now show that $C_\eps$ is bounded independently of $\eps$. Indeed, to the contrary assume that  $C_\eps$ is not bounded as $\epsilon\to 0$. Then we can find a subsequence $u_{\epsilon}$ such that
$$||u_\eps||_{H^1(D)} \geq \gamma_\epsilon \left(||f_\eps||_{H^{3/2}(\partial D)} +||g_\eps||_{H^{1/2}(\partial D)} \right)$$
and  $\gamma_\epsilon \rightarrow \infty$ as $\eps \rightarrow 0$. Let us define the sequences $\tilde{u}_\eps:= \frac{u_\eps}{||u_\eps||_{H^1(D)}}$,  $\tilde{f}_\eps:= \frac{f_\eps}{||u_\eps||_{H^1(D)}}$ and $\tilde{g}_\eps:= \frac{g_\eps}{||u_\eps||_{H^1(D)}}$. Hence we have that $(\tilde{f}_\eps , \tilde{g}_\eps) \rightarrow (0,0)$ as $\eps \rightarrow 0$ and  $({\mathbb I}-k^2{\mathbb K}_\epsilon)(\tilde u_\epsilon)=\tilde\alpha_\epsilon$. Hence
\begin{align*}
|| \tilde{u}_\eps ||_{H^2(D)} &\leq  k^2 ||\mathbb{K}_\eps(\tilde{u}_\eps)||_{H^2(D)}+ ||\tilde{\alpha}_\eps||_{H^2(D)}, \\
									&\leq M_1 || \tilde{u}_\eps ||_{H^1(D)}+M_2 \left(|| \tilde{f}_\eps ||_{H^{-3/2}(\partial D)} + || \tilde{g}_\eps ||_{H^{1/2}(\partial D)} \right)\leq M_1+M_2.
\end{align*}
Hence  $\tilde{u}_\eps$ is bounded and therefore has a weak limit in $H^2_0(D)$, which from Theorem \ref{th3.3} is a solution to the homogenized equation (\ref{39}) with zero boundary data. This implies that $\tilde{u}=0$ since $k^2$ is not a transmission eigenvalue for $\eps=0$ which contradicts the fact that $||\tilde{u}||_{H^1(D)}=1$, proving the result.
\end{proof}

We can now state the convergence result for the interior transmission problem. 

\begin{theorem}
Assume that either $n_{min}>1$ or $0<n_{max}<1$ and $k$ is not a transmission
eigenvalue for $\epsilon\geq 0$ small enough. Let $(w_\epsilon, v_\epsilon)\in
L^2(D)\times L^2(D)$ be such that $U_\eps=w_\epsilon-v_\epsilon \in H^2(D)$ is
a sequence of solutions to \eqref{4thitp} with $(f_\eps , g_\eps) \in
H^{3/2}(\partial D) \times H^{1/2}(\partial D)$ converging strongly  to
$(f,g)\in H^{3/2}(\partial D) \times H^{1/2}(\partial D)$ as $\epsilon\to
0$. Then  $U_\eps \weakc U$ in $H^2(D)$ and $(w_\eps , v_\eps) \weakc (w,v)$ in
$L^2(D) \times L^2(D)$ \big(strongly in $L^2_{loc}(D) \times
L^2_{loc}(D)$\big), where the limit $U$ satisfies 
\begin{eqnarray} \label{homA=1.1}
 \left(\Delta +k^2  n_h \right) \frac{1}{n_h-1} \left(\Delta +k^2  \right) U=0 \, \, &\textrm{ in }& \,  D \\\label{homA=1.2}
  U=f  \;\;\;\textrm{ and }  \;\;\; \frac{\partial U}{\partial \nu}=g &\textrm{ on }& \partial D,
 \end{eqnarray}
$U=w-v$, and $(w,v)$ satisfy
\begin{eqnarray}
 \Delta v +k^2 v=0 \;\;\; &\textrm{ and }& \;\;\; \Delta w +k^2 n_h w=0\;\; \;\; \textrm{ in } \;\;  D\\
 w-v=f\;\;\;  &\textrm{ and }&  \;\;\; \frac{\partial w}{\partial \nu}-\frac{\partial v}{\partial \nu}=g \;\; \;\; \;\; \;\; \;\; \textrm{ on }\;\; \partial D
 \end{eqnarray}
\end{theorem}
\begin{proof}
The result follows from combining Theorem \ref{th3.3} and Theorem \ref{th3.4}
and the uniqueness of solution for \eqref{homA=1.1}-\eqref{homA=1.2}.
\end{proof}
\section{Convergence of the Transmission Eigenvalues}
Using the convergence analysis for the solution of the interior transmission problem, we now prove the convergence of a sequence of real transmission eigenvalues of the periodic media, namely of  those who are bounded with respect to the small parameter $\epsilon$. The following lemmas provide conditions for the existence of real transmission eigenvalues that are bounded in $\epsilon$.
\begin{lemma}\label{lem4.1} The following holds:
\begin{enumerate}
\item Assume that $A_{\epsilon}=I$ for all $\epsilon>0$ and either $n_{min}>1$ or $0<n_{max}<1$. There exists an infinite sequence of real transmission eigenvalues $k^j_\epsilon$, $j\in{\mathbb N}$ of  (\ref{def1})-(\ref{def2}) accumulating at $+\infty$ such that
\begin{eqnarray*}
k^j(n_{max},D)\leq k^j_\epsilon<k^j(n_{min},D) \qquad &\mbox{if}&\; n_{min}>1\\
k^j(n_{min},D)\leq k^j_\epsilon<k^j(n_{max},D) \qquad &\mbox{if}&\; 0<n_{max}<1
\end{eqnarray*}
where $k^j(n,D)$ denotes an eigenvalue of (\ref{def1})-(\ref{def2}) with $A_\epsilon=I$ and $n_{\epsilon}=n$.
\item Assume that $n_{\epsilon}=1$ for all $\epsilon>0$ and either $A_{min}>1$ or $0<A_{max}<1$. There exists an infinite sequence of real transmission eigenvalues $k^j_\epsilon$, $j\in{\mathbb N}$ of  (\ref{def1})-(\ref{def2}) accumulating at $+\infty$ such that
\begin{eqnarray*}
k^j(a_{max},D)\leq k^j_\epsilon\leq k^j(a_{min},D) \qquad &\mbox{if}&\; a_{min}>1\\
k^j(a_{min},D)\leq k^j_\epsilon\leq k^j(a_{max},D) \qquad &\mbox{if}&\; 0<a_{max}<1
\end{eqnarray*}
where $k^j(a,D)$ denotes an eigenvalue of (\ref{def1})-(\ref{def2}) with $A_\epsilon=aI$ and $n_{\epsilon}=1$.
\end{enumerate}
Here $j$ counts the eigenvalue in the sequence under consideration which may not necessarily be the $j$-th transmission eigenvalue. In particular the first transmission eigenvalue satisfies the above estimates.
\end{lemma}
\begin{proof} The detailed proof of the above statements can be found in \cite{cakginhad}. We remark that the statements are not proven for all real transmission eigenvalues. For example in the case of first statement, from the proofs in  \cite{cakginhad}, real transmission eigenvalues are roots of $\lambda_j(\tau, n_\epsilon, D)-\tau=0$, where $\lambda_j$, $j=1\dots $, are eigenvalues of some auxiliary selfadjoint eigenvalue problem satisfying the Rayleigh quotient. The latter implies lower and upper bounds for $\lambda_j$ in terms of $n_{min}$ and $n_{max}$, and these bounds are also satisfied by  the transmission eigenvalues that are the smallest root of each  $\lambda_j(\tau, n_\epsilon, D)-\tau=0$. Same argument applies to the second statement also. In particular the estimates hold for the first transmission eigenvalue.
\end{proof}
The existence results and estimates on real transmission eigenvalues are more restrictive for the case when both $A_\epsilon \neq I$ and $n_\epsilon \neq 1$. The following result is proven in \cite{cakonikirsch} (see also \cite{CC-book}).
\begin{lemma} \label{lem4.2} The following holds:
\begin{enumerate}
\item Assume that either $a_{min}>1$ and $0<n_{max}<1$ or $0<a_{max}<1$ and  $n_{min}>1$. There exists a infinite sequence of real transmission eigenvalues $k^j_\epsilon$, $j\in{\mathbb N}$ of  (\ref{def1})-(\ref{def2}) accumulating at $+\infty$ satisfying
\begin{eqnarray*}
\qquad k^j(a_{max}, n_{min},D)\leq k^j_\epsilon<k^j(a_{min},n_{max}, D) \;&\mbox{if}&\; a_{min}>1, \;  0<n_{max}<1\\
\qquad k^j(a_{min},n_{max}, D)\leq k^j_\epsilon<k^j(a_{max},n_{min},D) \; &\mbox{if}&\; 0<a_{max}<1, \;  n_{min}>1
\end{eqnarray*}
where $k^j(a, n,D)$ denotes an eigenvalue of (\ref{def1})-(\ref{def2}) with $A_\epsilon=aI$ and $n_{\epsilon}=n$.
\item Assume that  $a_{min}>1$ and   $n_{min}>1$ or $0<a_{max}<1$ and  $0<n_{max}<1$. There exists finitely many real transmission eigenvalues $k^j_\epsilon$, $j=1\cdots p$ of  (\ref{def1})-(\ref{def2}) provided that $n_{max}$ is small enough. In addition they satisfy
\begin{eqnarray*}
\qquad 0< k^j_\epsilon<k^j(a_{min}/2, D) \;&\mbox{if}&\; a_{min}>1, \;  n_{min}>1\\
\qquad 0< k^j_\epsilon<k^j(a_{max}/2, D) \, &\mbox{if}&\, 0<a_{max}<1, \;  0<n_{max}<1
\end{eqnarray*}
where $k^j(a, D)$ denotes an  eigenvalue of (\ref{def1})-(\ref{def2}) with $A_\epsilon=aI$ and $n_{\epsilon}=1$.
\end{enumerate}
Here $j$ counts the eigenvalue in the sequence under consideration which may not necessarily be the $j$-th transmission eigenvalue. In particular the first transmission eigenvalue satisfies the above estimates.
\end{lemma}
\begin{proof}
The estimates follow by the same argument  as in the proof of Lemma \ref{lem4.1} combined with the existence proofs in \cite{cakonikirsch}. In  particular, the estimates can be obtained by modifying the proof of Theorem 2.6 and Theorem 2.10 in \cite{cakonikirsch} in a similar way as in the proof of Corollary 2.6 in  \cite{cakginhad}.
\end{proof}

\subsection{The case of $A_\eps \neq I$} We assume that  $A_{min}>1$ or $A_{max}<1$  in addition to (\ref{c0}) and (\ref{c1}) and let $k_{\epsilon}$ be one of the  transmission eigenvalues  described in Lemma \ref{lem4.1} and Lemma \ref{lem4.2}.  In particular  $\{k_\epsilon\}$ is bounded and hence there is a positive number $k\in{\mathbb R}$ such that $k_\epsilon\to k$ as $\epsilon \to 0$. Let $(w_\epsilon ,v_\epsilon)$  be a corresponding pair of eigenfunctions  normalized such that $|| w_\eps||_{L^2(D)}+|| v_\eps||_{L^2(D)}=1$. Notice from Section \ref{sec3.1} that the transmission eigenfunctions satisfy 
$$\mathcal{A}_\eps \left( (w_\eps , v_\eps); (\varphi_1,\varphi_2) \right)=0   \,\quad  \text{ for all } \, \,   (\varphi_1,\varphi_2) \in X(D)$$
where the sesquilinear form ${\mathcal A}_\epsilon(\cdot ; \cdot)$ is given by
\begin{eqnarray*}
\mathcal{A}_\eps \big( (w_\eps , v_\eps); (\varphi_1,\varphi_2) \big):=\int\limits_D A_\eps \grad w_\eps \cdot \grad \overline{\varphi}_1 -\grad v_\eps \cdot \grad \overline{\varphi}_2 -k^2_\eps(n_\eps w_\eps \overline{\varphi}_1- v_\eps \overline{\varphi}_2 )\, dx.
\end{eqnarray*}
Obviously if  $\mathbb{T}: X(D) \mapsto X(D)$ is a continuous bijection then we have that the  pair of the eigenfunction $(w_\epsilon ,v_\epsilon)$ satisfies
\begin{equation}\label{eigenf}
\mathcal{A}_\eps \big( (w_\eps , v_\eps); \mathbb{T}(w_\eps , v_\eps) \big)=0.
\end{equation}
We will use (\ref{eigenf}) to prove that the sequence $(w_\eps , v_\eps)$ is bounded in $X(D)$. To do so we must control the norm of the gradients of the functions in the sequence. Indeed, assuming that $A_{min}>1$ and letting $\mathbb{T}(w,v)=(w, -v+2w)$ gives that
\begin{eqnarray}\label{plaka}
\int\limits_D A_\eps \grad w_\eps \cdot \grad \overline{w}_\eps + |\grad v_\eps|^2 -2 \grad v_\eps \cdot \grad \overline{w}_\eps  \,dx= k^2_\eps \int\limits_D n_\eps |w_\eps|^2 + |v_\eps|^2 -2 v_\eps \overline{w}_\eps \, dx,
\end{eqnarray}
which by using Young's inequality gives that $||\grad w_\eps||^2_{L^2(D)}+||\grad v_\eps||^2_{L^2(D)}$ is bounded independently of $\epsilon>0$. Similarly in the case when $0<A_{max}<1$ we obtain the result using  $\mathbb{T}(w,v)=(w-2v, -v)$.

Therefore, in both cases  we have that $(w_\eps , v_\eps)$ is a bounded
sequence in $X(D)$. This implies that there is a  subsequence, still denoted by
$(w_\eps , v_\eps)$,  that converges weakly (strongly in $L^2(D) \times L^2(D)$
to some $(w,v)$ in $X(D)$).  The $L^2$-strong limit implies that $|| w
||_{L^2(D)}+|| v ||_{L^2(D)}=1$ hence $(w,v) \neq (0,0)$. Using similar
argument as at the beginning of Section \ref{sec3.1}  we have that $k$ is a
transmission eigenvalue, with  $(w,v)$  in $X(D)$  the corresponding transmission eigenfunctions, for the homogenized transmission eigenvalue problem
\begin{eqnarray}
\grad \cdot A_h \grad w +k^2 n_h w=0 \, \, \, & \textrm{ and }& \, \, \,  \Delta v + k^2 v=0  \, \, \textrm{ in } \,  D, \label{homogtevprob1} \\
w=v  &\textrm{ and }&  \frac{\partial w}{\partial \nu_{A_h}}=\frac{\partial v}{\partial \nu} \, \,  \textrm{ on }\,  \partial D. \label{homogtevprob2}
\end{eqnarray}
Hence we have proven the following result for the transmission eigenvalue problem.
\begin{theorem}\label{thmains}
Assume that  either $A_{min}>1$ or $0<A_{max}<1$ and let $k_\eps$ be a sequence of transmission eigenvalues for (\ref{def1})-(\ref{def2}) with  corresponding eigenfunctions $(w_\eps , v_\eps)$. Then, if $k_\epsilon$ is bounded with respect to $\epsilon$, then  there is a subsequence of $\{ (w_\eps , v_\eps) \, , k_\eps \} \in X(D) \times \R$ such that $(w_\eps , v_\eps) \weakc (w,v)$ in $X(D)$ \Big(strongly in $L^2(D) \times L^2(D)$\Big) and $ k_\eps \rightarrow k$ as $\eps \rightarrow 0$, where $\{ (w , v) \, , k\} \in X(D) \times \R$ is an eigenpair for \eqref{homogtevprob1}-\eqref{homogtevprob2}. 
\end{theorem}
\subsection{The case of $A_\eps \equiv I$} In this case we assume that either $n_{min}>1$ or $0<n_{max}<1$. Let $k_\epsilon$ be an eigenvalue of  (\ref{def1})-(\ref{def2}) with corresponding eigenfunctions  $ (w_\eps , v_\eps)  \in L^2(D) \times L^2(D)$ such that $u_\eps =w_\eps-v_\eps \in H^2_0(D)$. As discussed in Section \ref{sec3.2},  $(w_\eps , v_\eps)$ are distributional solutions to:
\begin{eqnarray}
  \Delta v_\eps +k^2_\eps v_\eps=0 \, \, \textrm{ and } \, \, \Delta w_\eps +k^2_\eps n_\eps w_\eps=0 \, \, \textrm{ in } \,  D,  \label{case3pde}
\end{eqnarray}
whereas $u_\eps \in H^2_0(D)$ solves
\begin{eqnarray}
0  &=& \left(\Delta +k^2  n_\eps \right) \frac{1}{n_\eps-1} \left(\Delta +k^2  \right) u_\eps \, \, \textrm{ in } \,  D, \label{4thitpn}
\end{eqnarray}
which  in the variational form reads
\begin{equation} \label{vareig}
\int\limits_D  \frac{1}{ n_\eps-1} \left( \Delta u_\eps +k^2_\eps u_\eps \right) \left( \Delta \overline{\varphi} +k^2_\eps n_\eps \overline{\varphi} \right) \, dx =0 \, \, \, \textrm{for all}  \, \, \,   \varphi \in H^2_0(D).
\end{equation}
We recall that $w_\epsilon, v_\epsilon$ and $u_\epsilon$ are related by
\begin{eqnarray}
v_\eps &=& - \frac{1}{k^2 (n_\eps-1)} \left( \Delta u_\eps +k^2   n_\eps u_\eps \right) \, \, \textrm{ in } \,  D \label{dvn}\\
w_\eps &=& - \frac{1}{k^2 (n_\eps-1)} \left( \Delta u_\eps +k^2  u_\eps \right) \, \, \textrm{ in } \,  D. \label{dwn}
\end{eqnarray}
Without loss of generality we consider the first real transmission eigenvalue $k_\epsilon:=k^1_\epsilon$ and set $\tau_\epsilon:=(k_\epsilon)^2$. Since the  corresponding eigenfunctions are nontrivial we can take $||u_\eps ||_{H^1(D)}=1$, and in addition we have the existence of a limit point $\tau$ for the set $\{\tau_\eps\}_{\eps >0}$.  Similarly to the previous case we wish to show that the normalized sequence $u_\eps$ is  bounded in $H^2_0(D)$. We start with the case when $n_{min}>1$ and let  $\frac{1}{n_{max}-1}=\alpha >0$. Taking $\varphi=u_\epsilon$ in (\ref{vareig}) implies 
$$ \int\limits_D  \frac{1}{ n_\eps-1} \left| \Delta u_\eps \right|^2 +\frac{2 \tau_\eps }{n_\eps-1}\Re (\Delta u_\eps \overline{u_\eps}) +\frac{\tau_\eps^2 n_\eps }{n_\eps-1} |u_\eps|^2\, dx =0 $$
Therefore, making use of Lemma \ref{lem4.1} part 1,  we have that:
$$\left|  \int\limits_D \frac{2 \tau_\eps}{n_\eps-1} (\Delta u_\eps ) \overline{u_\eps} \, dx \right| \leq  \frac{2  \tau(n_{min} , D) }{n_{min}-1}  \left|\int\limits_D (\Delta u_\eps) \overline{u_\eps} \, dx \right|  \leq \frac{2  \tau(n_{min} , D) }{n_{min}-1}  \int\limits_D |\grad u_\eps|^2 \, dx.$$
Which gives that: 
\begin{eqnarray*}
\alpha  ||\Delta u_\eps||^2_{L^2(D)}   \leq  \frac{ \tau(n_{min} , D)^2 n_{max}}{n_{min}-1} || u_\eps ||^2_{L^2(D)} + \frac{2  \tau(n_{min} , D) }{n_{min}-1}  || \grad u_\eps ||^2_{L^2(D)}.
\end{eqnarray*}

Now since $|| u_\eps ||_{H^1(D)}=1$ and using that $|| \Delta \cdot||_{L^2(D)}$ is an equivalent norm on $H^2_0(D)$ we have that $u_\eps$ is a bounded sequence. By the construction of $(w_\eps , v_\eps)$ we have that this is a bounded sequence in $L^2(D) \times L^2(D)$. Note that a similar argument holds if $0<(n_{max}-1)<1$, by multiplying the variational form by $-1$. Now by similar argument as in the proof of Theorem \ref{th3.3} we can now conclude the following result.
\begin{theorem}\label{thmains2}
Assume that $A_\epsilon \equiv I$ for all $\epsilon>0$ and either $n_{min}>1$ or $n_{max}<1$, and furthermore let $k_\eps$ be a  transmission eigenvalue for (\ref{def1})-(\ref{def2})  with  corresponding eigenfunctions $(w_\eps , v_\eps)$. Then, if $k_\epsilon$ is bounded with respect to $\epsilon$, there is a subsequence of $\{ (w_\eps , v_\eps) \, , k_\eps \} \in (L^2(D) \times L^2(D)) \times \R_+$ such that $(w_\eps , v_\eps) \weakc (w,v)$  in $L^2(D) \times L^2(D)$ and $ k_\eps \rightarrow k$ as $\eps \rightarrow 0$, where $\{ (w , v) \, , k \} \in \left(L^2(D) \times L^2(D) \right) \times \R_+$ is an eigenpair corresponding to 
$$ \Delta v +k^2 v=0 \, \, \textrm{ and } \, \, \Delta w + k^2 n_h w=0 \, \, \textrm{ in } \,  D, \quad w-v\in H^2_0(D).$$
\end{theorem}

\bigskip
The proofs of both Theorem \ref{thmains} and Theorem \ref{thmains2} simply depend on the boundedness of  the sequence of any real transmission eigenvalue in terms of $\epsilon$, therefore the proofs hold for all the eigenvalues that satisfy bounds stated in Lemma \ref{lem4.1} and Lemma \ref{lem4.2}.
\begin{remark} 
{\em  The transmission eigenvalues of the limiting problem \eqref{homogtevprob1}-\eqref{homogtevprob2} satisfy the same type of estimates as in  Lemma \ref{lem4.1} and Lemma \ref{lem4.2}. Furthermore, from the proof of Theorem \ref{thmains} and Theorem \ref{thmains2} one can see that the limit $k$ of the sequence $\{k_\epsilon\}$, where each $k_\epsilon$ is the first transmission eigenvalue of (\ref{def1})-(\ref{def2}), is the first transmission eigenvalue of \eqref{homogtevprob1}-\eqref{homogtevprob2}.
}
\end{remark}
\section{Numerical Experiments}
We start  this section with a preliminary numerical investigation on the convergence of the first transmission eigenvalue. To this end, we fix an $A_\eps$ and $n_\eps$ and investigate the behavior of the first transmission eigenvalue $k_1(\eps)$ on $\epsilon$. More specifically, we investigate the convergence rate of $k_1(\eps)$ to the first eigenvalue $k_h$ corresponding to the homogenized problem. The first transmission eigenvalue for the periodic media and homogenized problem is computed using a mixed finite element method with an eigenvalue-searching technique described in \cite{sun} and \cite{sun1}. In addition, we show numerical examples  of determining  the first few real transmission eigenvalues from the far field scattering data. This section is concluded with some examples demonstrating that the first real transmission eigenvalue  provides  information about the effective material properties $A_h$ and $n_h$ of the periodic media.  
\subsection{Numerical Tests for the Order of Convergence}
We consider the case where the domain $D=B_R$ with $R=2$ and for the first example assume that the periodic media is isotropic, i.e. $A_\epsilon =I$,  with refractive index
$$n_\eps= \sin^2(2\pi x_1/\eps)+\cos ^2(2\pi x_2/\eps)+ 2.$$
Obviously  $n_h=3$. If the domain is a ball of radius two in $\R^2$ separation of variables gives that the roots of 
$$d_0(k)=\text{J}_0\left(2k \sqrt{n_h}\right)\text{J}_1(2k)-\sqrt{n_h}\text{J}_1\left(2k \sqrt{n_h}\right)\text{J}_0(2k)$$
are transmission eigenvalues. Using the secant method we see that $k_h \approx 2.0820$.  The values of the first transmission eigenvalue for the periodic media for different values of $\epsilon$ are shown in Table \ref{tt1}.\\
\begin{table}[h!]
\centering  
\begin{tabular}{c c c c c c c c c c } 
\hline\hline                        
$\eps$    &  1/3 & 1/4 & 1/5 & 1/6 & 1/7 \\ [0.5ex] 
\hline                  
 $k_1(\eps)$ &2.0842 &2.0834 & 2.0829 & 2.0828 & 2.0824 \\  [1ex]  
\hline 
\end{tabular}
\caption{First TEV for various $\eps$ with $A_\eps = I$ and $n_\eps \neq 1$ }\label{tt1}
\end{table}

To find the convergence rate we assume that the error satisfies that 
$$|k_1(\eps)-k_h|=C \eps^p \, \, \text{ which gives } \, \, \log \big( |k_1(\eps)-k_h|  \big) = \log(C) + p \log( \eps )$$
for some constant $C$ independent of $\eps$. Using the $\texttt{polyfit}$ command in Matlab we can find a $p$ that approximately satisfies the above equality. The calculations give that in this case $p=2.1486$ (see Figure \ref{ff1}). 
\begin{figure}[H]
\begin{center}
\includegraphics[scale=0.35]{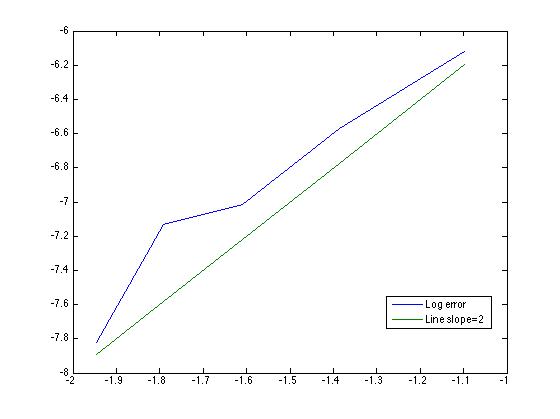}\\
\caption{Here is a $\texttt{Log-Log}$ plot that compares the $\log \big|k_1(\eps)-k_h \big|$ to the line with slope 2. }\label{ff1}
\end{center}
\end{figure}
In the next example we keep the same domain $D$ and take the periodic constitutive parameters of the media 
\begin{equation}
n_\eps= \sin^2(2\pi x_1/\eps)+\cos ^2(2\pi x_2/\eps)+ 2
\end{equation}
and
\begin{equation}\label{aep}
 A_\eps=\frac{1}{3} \begin{pmatrix}  \sin^2(2\pi x_2/\eps)+1 & 0 \\ 0 &  \cos^2(2\pi x_1/\eps)+1 \end{pmatrix}.
 \end{equation}
Notice that $\grad_y \cdot A e_i=0$ which gives that $A_h=\frac{1}{2} I$ and $n_h=3$. In this case the first zero of 
$$d_0(k)=\text{J}_0\left(2k \sqrt{\frac{n_h}{A_h}}\right)\text{J}_1(2k)-\sqrt{n_h A_h}\text{J}_1\left(2k \sqrt{\frac{n_h}{A_h}}\right)\text{J}_0(2k)$$
is the first  transmission eigenvalue $k_h^1$ for the homogenized problem which turns out to be $k^1_h=1.0582$. Similarly we use $\texttt{polyfit}$ in Matlab to find a $p$ such that $ \log \big( |k_1(\eps)-k_h|  \big) = \log(C) + p \log( \eps )$. In this case we calculate that $p=1.4421$.  The results are shown in Table \ref{tt5} and Figure \ref{ff5}.

\begin{table}[h!]
\centering  
\begin{tabular}{c c c c c c c} 
\hline\hline                        
$\eps$    &  1 & 1/2 & 1/3 & 1/4 & 1/5 & 1/6 \\ [0.5ex] 
\hline                  
 $k_1(\eps)$ & 1.0592 & 1.0591 & 1.0587 & 1.0586 & 1.0584 & 1.0583 \\  [1ex]  
\hline 
\end{tabular}
\caption{First TEV for various $\eps$ with $A_\eps \neq I$ and $n_\eps \neq 1$ } \label{tt5}
\end{table}
\begin{figure}[H]
\begin{center}
\includegraphics[scale=0.35]{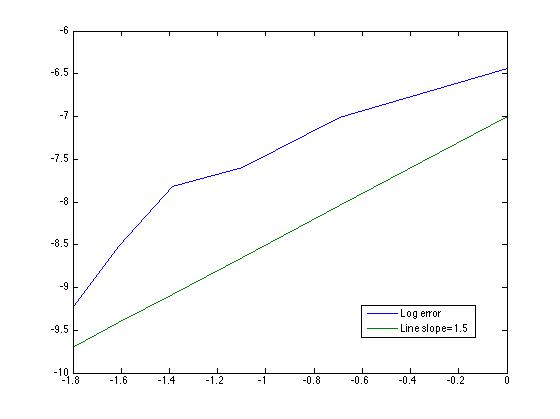}\\
\caption{Here is a $\texttt{Log-Log}$ plot that compares the $\log \big|k_1(\eps)-k_h \big|$ to the line with slope  } \label{ff5}
\end{center}
\end{figure}
In these two examples the convergence rate seems to be better than of order $\epsilon$. Notice that the boundary correction in these both cases does not appear since there is no boundary correction if $A=I$ and in the second example we have $\grad_y \cdot A_\eps e_i=0 \Longrightarrow \psi(y)=0$ which yield no boundary correction (this will become clear in the second part of this study but for the case of Dirichlet and Neumann conditions see \cite{shari1} and \cite{shari2}, respectively). We now wish to investigate the numerical convergence rate when $\psi(y) \neq 0$.  Hence take 
\begin{equation}\label{nep}
n_\eps= \sin^2(2 \pi x_1/\eps)+2
\end{equation}
  and  $\tilde A_\eps=T A_\eps T^{\top}$ where $A_\epsilon$ is given by (\ref{aep}) and $T$ is the matrix representing clockwise rotation by $1$ radian. We now compute the first transmission eigenvalue
with coefficients $n_\epsilon$ and $\tilde A_\eps$. Since now $\psi(y) \neq 0$,
we cannot compute analytically $A_h$ (one need to solve  the cell problem
numerically in order to compute $A_h$) and hence we do not have a value for the
first transmission eigenvalue of the homogenized problem. In this case, in
order to obtain an idea about the convergence order of the first
transmission eigenvalue we define the relative error as:
$$\text{R.E.}= \frac{|k_1(\eps)-k_1(\eps/2)|}{k_1(\eps/2)}$$
and find the convergence rate for the relative error is a similar manner as
discussed above. The Table \ref{tt7} and Figure \ref{ff7} show  the computed
first  transmission eigenvalue for various epsilon in the square $D:=[0,
2]\times [0, 2]$ and the circular domain $D:=B_R$ of radius $R=1$.
\begin{table}[h!]
\centering  
\begin{tabular}{c c c c c c c} 
\hline\hline                        
$\eps$    &  1 & 1/2 & 1/4 & 1/8 & \vspace{0.1in}  | &Convergence Rate \\ [0.5ex] 
\hline                  
 Circle $k_1(\eps)$ & 2.460 & 2.453 & 2.472 & 2.518 & \vspace{0.1in} | & 1.32 \\          
\hline 
 Square $k_1(\eps)$ & 2.201 & 2.213 & 2.230 & 2.273 &  \vspace{0.1in} | & 0.917 \\  [1ex]  
 \hline 
\end{tabular}
\caption{First TEV for various $\eps$ shown in the first row corresponding to $\tilde A_\eps$ and $n_\eps$. Last column shows the convergence rate.}\label{tt7}
\end{table}
\begin{figure}[H]
\begin{center}
\includegraphics[scale=0.32]{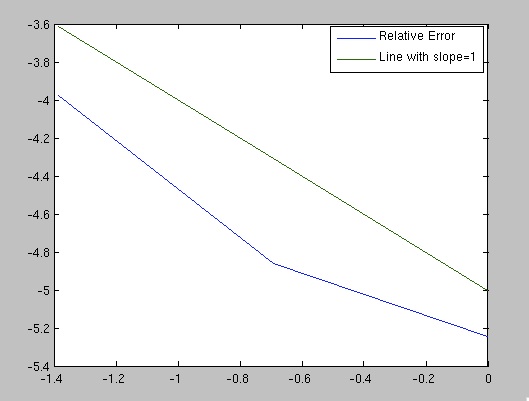}\includegraphics[scale=0.3145]{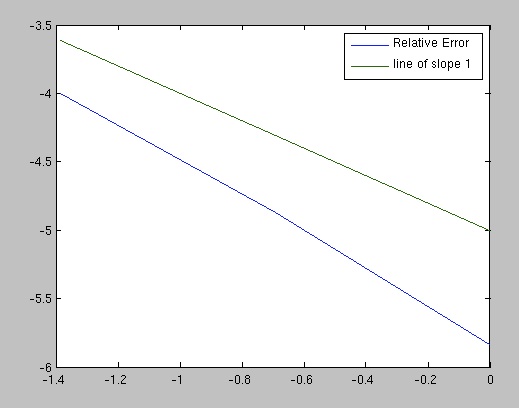}\\
\caption{Convergence graph for relative error when $\psi(y) \neq 0$ compare to the line with slope one. 
On the left we have the $\texttt{Log-Log}$ plot for the square and on the right for the disk.}\label{ff7}
\end{center}
\end{figure}
The above results seem to suggest that the relative error is of order $\eps$. In this case the boundary corrector is non-zero which explain this order of convergence.
\subsection{Transmission Eigenvalues and the Determination of  Effective Material Properties}
For the given inhomogeneous media, the corresponding transmission eigenvalues are closely related to the so-called non-scattering frequencies, i.e.  the values of $k$  for which there exists an incident wave doesn't scatter \cite{corner}, \cite{CakHad2}. The  scattering problem associated with our  transmission eigenvalue problem in ${\mathbb R}^2$ is given by
\begin{eqnarray*}
\grad \cdot A(x/\eps) \grad w_\eps +k^2 n(x/ \eps) w_\eps=0 \,  &\textrm{ in }& \,  D\\
\Delta u^s_\eps+k^2 u^s_\eps=0  &\textrm{ in }&  \R^2 \setminus \overline{D} \\
 w_\eps-u^s_\eps=u^i \, \, \, \textrm{ and }\, \, \, \frac{\partial  w_\eps}{\partial \nu_{A_\eps}}- \frac{\partial u^s_\eps}{\partial \nu} =\frac{\partial u^i}{\partial \nu} &\textrm{ on }& \partial D \\
\lim\limits_{r \mapsto \infty} \sqrt{r}\left( \frac{\partial u^s_\eps}{\partial r} -iku^s_\eps \right)=0 
\end{eqnarray*}
The asymptotic behavior of $u^s_\eps(r,\theta)$ can be shown to be \cite{CC-book}
\begin{equation*}
u^s_\eps(r,\theta)=\frac{e^{ikr}}{\sqrt{r}}u^{\infty}_\eps(\theta,\phi) + \mathcal{O} \left( r^{-3/2} \right) \; \textrm{  as  } \;  r \to \infty. 
\end{equation*}
where the function $u^{\infty}_\eps$ is called the far field pattern of the scattering problem with incident direction $\phi$ and observation angle $\theta$. 
Recall that the far field operator $F_\eps: L^2(0, 2\pi) \mapsto  L^2(0, 2\pi)$ is defined by
$$(F_\eps g)(\theta):=\int\limits_0^{2\pi} u^{\infty}_\eps(\theta,\phi) g(\phi) \, d\phi .$$
It has been shown that the transmission eigenvalues can be determine form a knowledge of the far field operator $F_\epsilon$ \cite{blow} and \cite{kirsch}. Now we would like to investigate how the first transmission eigenvalue determined from the far field operator depends on the parameter $\epsilon$. Here  to find the transmission eigenvalues from the far field data, we follow the approach in \cite{blow}. To this end, let $\Phi_{\infty}(\cdot,\cdot)$ be the  far field pattern for the fundamental solution to the Helmholtz equation. If $g_{z,\delta}$ is the Tikhonov  regularized solution of the far field equation, i.e.  the unique minimizer of the functional: 
$$\|F_\eps g-\Phi_{\infty}(\cdot, z)\|^2_{L^2(0,\,2\pi)}+\alpha \|g\|^2_{L^2(0,\,2\pi)}$$
with the regularization parameter $\alpha:=\alpha(\delta)\to 0$ as the noise level $\delta\to 0$, then at a transmission eigenvalue $\|v_{g_{z,\delta}}\|_{L^2(D)}\to \infty$ as $\delta\to 0$ for almost every $z\in D$, whereas otherwise bounded, where $v_g(x):=\int_{0}^{2\pi}g(\phi)e^{ik(x_1\cos\phi+x_2\sin\phi)}\,d\phi$. To compute the simulated data we use a FEM method to approximate the far field pattern corresponding to the scattering problem. Using the approximated $u^{\infty}_\eps$  we then solve: $F_\eps g=\Phi_{\infty}(\cdot,z)$ for 25 random values of $z\in D$ where the regularization parameter is chosen based on Morozov's discrepancy principle. The transmission eigenvalues will appear as spikes in the plot  of $||g_z||_{L^2[0,2\pi]}$ versus $k$. In our example we choose the domain $D:=B_R$ to be the ball of radius $R=1$ and the material properties $n_\epsilon$ given by (\ref{nep}) and $A_\epsilon$ given by (\ref{aep}). The effective material properties are $A_h=\frac{1}{2} I$ and $n_h=\frac{3}{2}$ and the corresponding first transmission eigenvalue is $k^1_h=2.5340$.  The computed transmission eigenvalue for this configuration for the  choices of $\epsilon =1$ and $\epsilon=0.1$ are shown in Figure \ref{ff99}
\begin{figure}[H]
\begin{center}
\includegraphics[scale=0.32]{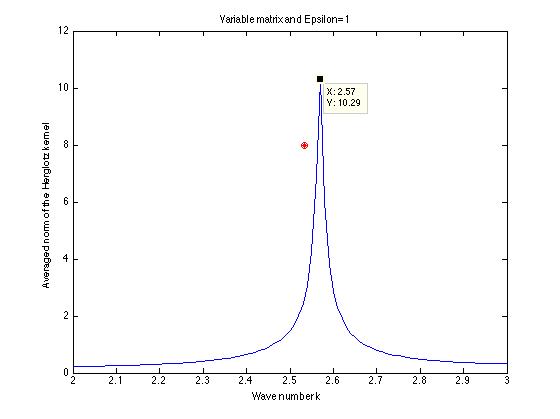}\includegraphics[scale=0.32]{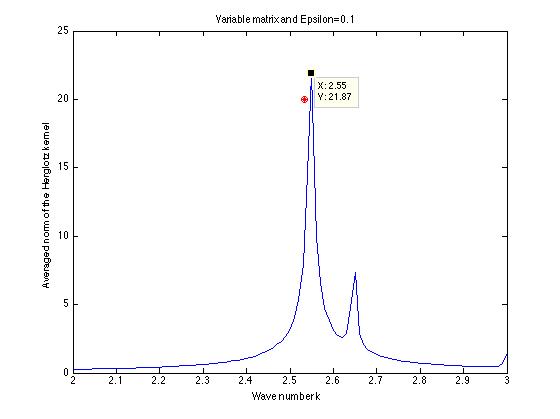}\\
\caption{On the left $\eps=1$ and on the right is $\eps=0.1$. The red dot indicates $k_h^1$ whereas the pick indicates $k_\epsilon^1$.} \label{ff99}
\end{center}
\end{figure}

The measured first transmission eigenvalue can be used to obtain information about the effective material properties $A_h$ and $n_h$. If $A_\epsilon=I$, it is known that $k_h^1$ uniquely determines $n_h$ and also the transmission eigenvalue depend continuously on $n_h$ \cite{drossos,coss,haddex}. From the scattering data we measure $k_\epsilon^1$ which for epsilon small enough is close to $k_h^1$.
Hence having available $k_\epsilon^1$ we find a constant $n$ such that the first transmission eigenvalue of the homogeneous media with refractive index $n$ has $k_\epsilon^1$ as the first transmission eigenvalue. Then by continuity this constant $n$ is close to $n_h$.  In Table \ref{tt99} we show the calculations for the ball of radius $D$, $A_\epsilon=I$ and $n_\eps=n(x / \eps)= \sin^2(2 \pi x_1/\eps)+2$. 
\begin{table}[H]
\centering  
\begin{tabular}{c c c c  } 
\hline\hline                        
$\eps$    &  $k_{\eps,1}$ & $n_h$ & reconstructed $n_h$   \\ [0.5ex] 
\hline                  
 0.1 & 5.046 & 2.5 & 2.5188  \\ [1ex]  
 \hline 
\end{tabular}
\caption{Reconstruction of $n_h$}\label{tt99}
\end{table}

Similarly, we can obtain information about the effective constant matrix $A_h$ \cite{ccm09}, \cite{9}. In particular, in the case when $n_\epsilon=1$, from the first transmission eigenvalue $k_h^1$ we can determine a constant $a$  which is in the middle of the smallest and the largest eigenvalues (in fact earlier numerical example suggest that this constant is roughly the arithmetic average of the eigenvalues of $A_h$). As an example we again consider the ball $D:=B_R$ of radius $R=1$, $n_\epsilon=1$ and  $A_\epsilon$ given by (\ref{aep}). Then having the measured $k^1_{\eps}$, we find the constant $a$ such that the first eigenvalue of the homogeneous media with $A=aI$ and $n=1$ is equal to  $k^1_{\eps}$. The calculation are shown in  Table \ref{tt88}.
\begin{table}[H]
\centering  
\begin{tabular}{c c c c  } 
\hline\hline                        
$\eps$    &  $k^1_{\eps}$ & $A_h$ & reconstructed $A_h$   \\ [0.5ex] 
\hline                  
 0.1 & 7.349 & 0.5$I$ & 0.4851$I$  \\ [1ex]  
 \hline 
\end{tabular}
\caption{Reconstruction of affective material property from FFE in unit disk } \label{tt88}
\end{table}
In the above both examples we see that the measured first transmission eigenvalue corresponding to the periodic highly oscillating media  can accurately determine the effective isotropic material properties $A_h=a_hI$ or $n_h$.
Next we consider an example where $A_h$ is constant matrix. We take the ball $D:=B_R$ of radius $R=1$ and $n_\epsilon=1$  and $\tilde A_\eps=T A_\eps T^{\top}$ where $A_\epsilon$ is given by (\ref{aep}) and $T$ is the matrix representing clockwise rotation by $1$ radian. In this case it becomes non-trivial to compute $A_h$ (one needs to solve the cell PDE problem). However the constant $a$ found as in the above example is in between (roughly the average) of the smallest and the largest eigenvalue of $A_h$.  The results are shown in Table \ref{tt44}
\begin{table}[H]
\centering  
\begin{tabular}{c c c  } 
\hline\hline                        
$\eps$    &  $k_{\eps,1}$  & reconstructed $a$   \\ [0.5ex] 
\hline                  
 0.1 & 7.5499 & 0.4921$I$  \\ [1ex]  
 \hline 
\end{tabular}
\caption{Reconstruction for the unit disk and  $A_\epsilon$ given by  (\ref{aep})}\label{tt44}
\end{table}
Furthermore, if both  $A_\eps \neq I$ and $n \neq 1$ we use a similar method as the above  to obtain information about $A_h/n_h$ \cite{cakonikirsch}. Here we look for a constant  $\alpha$ such that the first eigenvalue of 
\begin{eqnarray*}
\Delta w +\alpha k^2 w=0 \, \, \, & \textrm{ and }& \, \, \,   \Delta v + k^2v=0 \, \, \, \,  \textrm{ in } \, \, \, \,  D  \\
w=v \, \, \,  &\textrm{ and }& \, \, \,  \frac{\partial w}{\partial \nu}=\frac{\partial v}{\partial \nu} \, \, \, \textrm{ on } \, \, \,\partial D 
\end{eqnarray*}
coincide with $k^1_{\eps}$ (note that here we incorrectly drop the jump in the normal derivative), where we take $n_\epsilon$ given by (\ref{nep}) and $A_\epsilon$ given by (\ref{aep})
giving that the ratio $\frac{n_h}{a_h}=5$.  The reconstruction is shown in Table \ref{tt33}.
\begin{table}[H]
\centering  
\begin{tabular}{c c c  } 
\hline\hline                        
$\eps$    &  $k_{\eps,1}$  & reconstructed $\frac{n_h}{a_h}$   \\ [0.5ex] 
\hline                  
 0.1 & 2.5415 & 4.788  \\ [1ex]  
 \hline 
\end{tabular}
\caption{Reconstruction of the ratio $\frac{n_h}{a_h}=5$ of effective material property for the  unit disk $D$}\label{tt33}
\end{table}
 
In all the examples so far we have considered smooth coefficients $A_\epsilon$
and $n_\epsilon$.  Hence, our next example concerns a checker board patterned
media where the coefficients take different values in the white and black
squares. Again here the scaled period for the coefficients is $Y=[0,1]^2$. The white and black squares are assumed to cover the same area in a unit cell. See Figure \ref{figu} for the definition of the coefficients. In this case we have that $n_h= 7/2$ and $A_h$ is shown in \cite{bojan} to be a scalar matrix, i.e. $A_h=a_hI$ where $a_h$ can be computed numerically.
\begin{figure}[H]
\begin{center}
\includegraphics[scale=0.35]{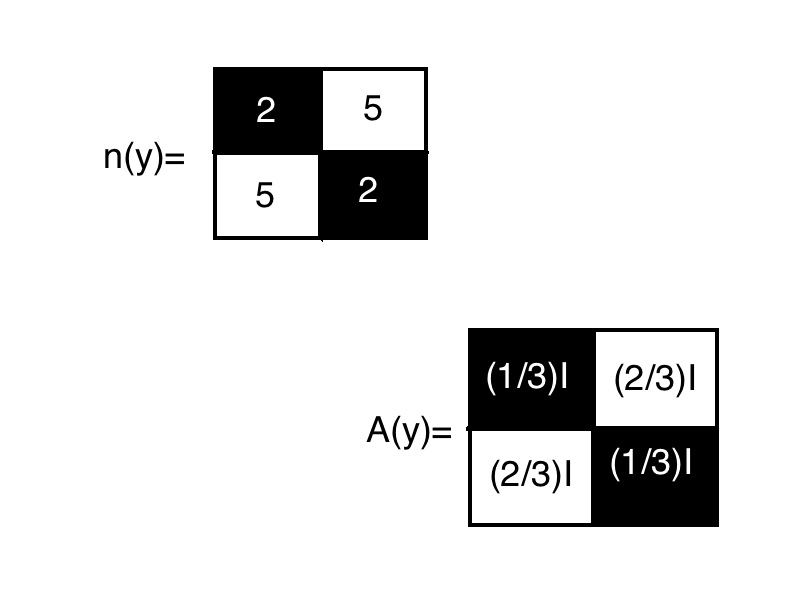}\\
\caption{Definition of Checker board coefficients. }\label{figu}
\end{center}
\end{figure}
See Table \ref{tt11} for a comparison between the first transmission eigenvalue of  the homogenized media and periodic media.
\begin{table}[H]
\centering  
\begin{tabular}{c c c c c c c c c} 
\hline\hline                        
$k_{1}(n(y))$&$k_{1}(n_h)$ & \vline & $k_{1}(A(y))$&$k_{1}(A_h)$& \vline& $k_{1}(n(y), A(y))$&$k_{1}(n_h, A_h)$   \\ [0.5ex] 
\hline                  
 1.0930 & 1.0757 & \vline & 1.9027 & 1.896 &\vline & 0.7673 & 0.7139  \\ [1ex]  
 \hline 
\end{tabular}
\caption{ Media with checkerboard pattern in $[-3,3]^2$}\label{tt11}
\end{table}
Next we use the first transmission eigenvalue for the actual media to determine the effective material properties. The result are shown in Table \ref{ball}
\begin{table}[H]
\centering  
\begin{tabular}{c c c c c c c c c} 
\hline\hline                        
$A(y)=I$, $n(y)$ & \vline & reconstructed $n_h=3.4123$ (exact $n_h=3.5$) \\ [0.5ex] 
\hline                  
$A(y)$, $n(y)=1$ & \vline & reconstructed $a_h=0.4472$  \\ [1ex]  
 \hline 
 $A(y)$, $n(y)$ & \vline & reconstructed $n_h/a_h=7.4704$  which gives $a_h=0.4685$\\ [1ex]  
 \hline 
\end{tabular}
\caption{Reconstructed of effective material properties for the checkerboard}\label{ball}
\end{table}

Lastly consider the case of a media with periodically spaced voids (subregions with $n_\eps=1$ and $A_\eps=I$). Our  analysis does not cover this  type of material property (see \cite{harris-thesis} for the case when $D$ is a union of cells) but nevertheless we consider an example of this type (The existence of real transmission eigenvalues for media with voids is proven in \cite{cch2010,harris}). In particular, we consider an example of isotropic media with  refractive index $A(y)=I$ and 
$$n(y) =\left\{ \begin{array}{rcl} 1 & \mbox{if} & (y_1-0.5)^2+(y_2-0.5)^2 < 0.25^2 \\ 5 & \mbox{if} &(y_1-0.5)^2+(y_2-0.5)^2 \geq 0.25^2  \end{array}\right.$$
 which gives that $n_h=5-\frac{\pi}{4}$, and  an example of anisotropic case with the same $n(y)$ and 
$$A(y) =\left\{ \begin{array}{rcl} I & \mbox{if} & (y_1-0.5)^2+(y_2-0.5)^2 < 0.25^2 \\ 0.5I & \mbox{if} &(y_1-0.5)^2+(y_2-0.5)^2 \geq 0.25^2  \end{array}\right.$$
where the period is $Y=[0,1]^2$ and  the is domain $D=[-3,3]^2$. See Table \ref{tlas} for the comparison of the first transmission eigenvalue for the homogenized media and the actual periodic media.
\begin{table}[H]
\centering  
\begin{tabular}{c c c c c } 
\hline\hline                        
$k_{1}(n(y))$    &  $k_{1}(n_h)$  & \vline & $k_{1}(n(y), A(y))$ & $k_{1}(n_h, A_h)$   \\ [0.5ex] 
\hline                  
 0.8745 & 0.8781 & \vline & 0.7599 & 0.7231  \\ [1ex]  
 \hline 
\end{tabular}
\caption{ Media with periodic voids in $[-3,3]^2$}\label{tlas}
\end{table}
In Table \ref{ball2} we show reconstructed effective material properties based on the first transmission eigenvalue. Note that $a_h$ is between the smallest and the largest eigenvalues of $A_h$.  
\begin{table}[H]
\centering  
\begin{tabular}{c c c c c c c c c} 
\hline\hline                        
$A(y)=I$, $n(y)$ & \vline & reconstructed $n_h=4.2678$ (exact $n_h=4.2146$) \\ [0.5ex] 
\hline                  
 $A(y)$, $n(y)$ & \vline & reconstructed $n_h/a_h=5.0550$  which gives $a_h=0.8337$\\ [1ex]  
 \hline 
\end{tabular}
\caption{Reconstructed effective material properties for the checkerboard}\label{ball2}
\end{table}
\bibliographystyle{plain}

\medskip


\medskip
 {\it E-mail address: }cakoni@math.udel.edu\\
 \indent{\it E-mail address: } haddar@cmap.polytechnique.fr\\
  \indent{\it E-mail address: } iharris@udel.edu\\
\end{document}